\documentclass{amsart}

\usepackage{color,amssymb}

\newcommand{\dt}{\partial_t}
\newcommand{\dn}{\partial_n}
\newcommand{\dz}{\partial_z}
\newcommand{\nam}{\nabla^\mu}

\newcommand{\R}{\mathbb R}
\newcommand{\T}{\mathbb T}
\newcommand{\Z}{\mathbb Z}
\newcommand{\eps}{}
\newcommand{\dsp}{\displaystyle}
\newcommand{\bott}{{\vert_{z=-1}}}

\newcommand{\trace}{{\vert_{Y=\frac{X}{\gamma}}}}
\newcommand{\sech}{{\rm sech}}
\newtheorem{lemm}{Lemma}[section]
\newtheorem{rema}[lemm]{Remark}
\newtheorem{prop}[lemm]{Proposition}

\newtheorem{thm}[lemm]{Theorem}
\newcommand{\abs}[1]{\vert#1\vert}
\newcommand{\Abs}[1]{\Vert#1\Vert}

\newcommand{\BigOh}[1]{\mathcal{O}(#1)}

\numberwithin{equation}{section}

\begin{document}

\title{Water waves over a rough bottom in the  
       shallow water regime}

\author{Walter Craig, David Lannes and Catherine Sulem}

\address{
Department of Mathematics\\
McMaster University \\
Hamilton, ON L8S 4K1, Canada }
\thanks{}
\email{craig@math.mcmaster.ca}

\address{D\'epartement de math\'ematiques et applications\\
Ecole normale sup\'erieure\\
45 rue d'Ulm\\
F-75230 Paris Cedex 05, France}
\thanks{}
\email{David.Lannes@ens.fr}

\address{
Department of Mathematics\\
University of Toronto\\
Toronto, ON M5S 2E4, Canada}
 \thanks{}
\email{sulem@math.toronto.edu}



\begin{abstract}
This is a study of the Euler equations for free surface water waves in
the case of varying bathymetry, considering the problem in the shallow
water scaling regime. In the case of rapidly varying periodic bottom
boundaries this is a problem of homogenization theory. In this setting
we derive a new model system of equations, consisting of the classical 
shallow water equations coupled with nonlocal evolution equations for 
a periodic corrector term. 
We also exhibit a new resonance phenomenon between surface waves and a 
periodic bottom. This resonance, which gives rise to secular growth of 
surface wave patterns, can be viewed as a nonlinear generalization of the 
classical Bragg resonance.
We justify the derivation of our model with 
a rigorous mathematical analysis of the scaling limit and the
resulting error terms. The principal issue is that the shallow water
limit and the homogenization process must be performed
simultaneously. Our model equations and the error analysis are 
valid for both the two- and the three-dimensional physical problems. 
\end{abstract}

\maketitle


\section{Introduction}
Studies of the Euler equations for free surface water waves are
important to understanding the dynamics of ocean waves. The case of
an idealized flat bottom and the resulting model equations has been
widely studied for many years. The more realistic
situation of varying bathymetry is less well known, despite its 
fundamental importance to studies of ocean wave dynamics in coastal
regions, and there is not a complete consensus as to the appropriate
model equations. In the case of topography there are 
many asymptotic scaling regimes of interest, including long-wave of
modulational hypotheses for the  evolution of the free surface, and
short scale and/or long scale variations in the variable bottom fluid 
boundary. 

In this paper we address the evolution of waves in the
shallow water regime, for which we investigate the effect of the
roughness of the bottom topography. The simplest situation is where
the bottom varies periodically and rapidly with respect to the typical
surface wavelength, a regime which can be described in the context of
homogenization theory. Ideally, wave motion in this regime of rapid
periodic bottom variations is described in terms of a long wave 
{\it effective} component, which is then adjusted by a smaller
multi-scale {\it corrector} at the next order of approximation. In
terms of the initial value problem for this regime, initial
configurations consisting of large scale data with a multi-scale
corrector term are expected to give rise to solutions with the same
character, up to a smaller error term. In this paper we derive a
system of model equations for such multi-scale approximate solutions. 
While other authors have looked at similar situations, as far as we
know this system is new, consisting of a 
version of the shallow water equations  for a mean field or effective 
components of the surface elevation and the fluid velocity, which 
then drive a nonlocal system of two additional equations 
for the evolution of a more rapidly oscillating corrector term. 
Because of the number of other models that have been proposed to describe
this setting, we justify the derivation of our system with a rigorous
analysis, giving error estimates for our approximate solutions.
In cases in which there is a resonance between
the effective velocity and the periodic bottom, the solution of the corrector
equation can exhibit secular growth at a linear rate.This phenomenon can be viewed as a nonlinear generalization of the classical Bragg resonance between the bottom topography and the free surface.
This is a local phenomenon, which may occur when the local Froude 
number is subcritical. In the absence of resonances, our analysis is valid over
time intervals of existence of the effective component.

The literature on models of free surface water waves over a variable
depth is extensive, including the paper of Miles \cite{Miles77} on
its Hamiltonian formulation, and that of Wu \cite{Wu94} on models 
which are valid in long wave scaling regimes. The paper of Rosales \&
Papanicolaou \cite{RP83} studies the long wave regime in which the
bottom is rapidly varying, in the sense that the typical wavelength of
surface waves is taken to be much longer than typical lengthscale of
the variations of the bottom depth. When the latter are periodic, or
more generally when they are given by a stationary ergodic process,
the techniques of homogenization theory are used to obtain effective
long wave model equations. The two most important examples are of
periodic bottom topography, and of topography given by a stationary 
random process. Recently there has been a renewal of interest in this
problem, both from the point of view of modeling of water waves in
asymptotic scaling regimes, and of mathematical analysis. A central
question is the validity of the homogenization approximation, and the
character of the resulting model equations. Following \cite{RP83}, the
paper of Nachbin \& S\o lna \cite{NS03} studies the deformation of
surface waves by the effects of propagation over a rough bottom, taken
in the shallow water scaling regime. In this work the bottom is given
by a random process, and the authors treat both the two- and
three-dimensional cases. The paper of Craig et al \cite{CGNS}
considers large periodic bottom variations, again for dimensions
$n=1+d$ ($d=1,2$), deriving model equations to quite high order of accuracy for
the profiles which describe weak limits of surface waves in the
homogenization limit of the nonlinear long wave regime. Similarly, the
paper of Garnier, Kraenkel \& Nachbin \cite{GKN07} studies the long
wave scaling limits of water waves over a periodic bottom (for $d=1$),
deriving an effective KdV equation, for which they describe the
dependence of the coefficients of nonlinearity and dispersion on the
topography of the bottom. This study continues in Garnier, Grajales \&
Nachbin  \cite{GGN07} in the case of random bathymetry. 
There are other studies of surface wave propagation over periodic
bathymetry, that focus on regimes which are not homogenization
theoretic. Namely, there is the case in which the typical wavelength
of surface waves is comparable or smaller than the typical bottom
variations. Among these, Choi \& Milewski \cite{CM02} consider
periodic solutions of systems of KdV equations which are coupled
through resonant interactions with a periodic bottom. The paper 
by Nahoulima et al \cite{NZPTK05} considers shallow water theory with
and without dispersive corrections, for a periodic and piecewise 
constant bottom of very long wavelength. 

The paper of Grataloup \& Mei \cite{GM03} considers the propagation of
modulational solutions over a random seabed in dimension $d=1$, which 
is extended to the case $d=2$ in Pihl, Mei \& Hancock \cite{PMH02}. In
this work, the typical wavelengths represented in the surface and the
topography are comparable, and the effort is to derive envelope
equations for the free surface and to understand its statistical
properties, given the ensemble of realizations of the random
bathymetry. 

There is also a long history of study of resonant interaction between water surface waves
with periodic bottom. The paper of  Mei \cite{MeiCC} gives the theory of linear Bragg 
resonances between surface waves and bottom variations of the same spatial scale. This
is extended to nonlinear resonances in Liu \& Yue \cite{LiuYue}.
The difference between these references and our work is that, in the latter,
 short scales perturbations of the free surface are generated by interaction of the bottom with long
waves on the free surface, a feature typical of homogenization theory.

None of the references above, however, give a mathematical
theorem which justifies on a rigorous basis the model equations that
are derived. After the derivation of the shallow water model in the
present paper, the second main point of our work is to provide a 
rigorous justification of this derivation. There 
is a history of results on the mathematical verification of the 
model equations for free surface water waves, starting in fact 
with the papers of Ovsjannikov~\cite{Ov74,Ov76} and 
Kano \& Nishida~\cite{KN79} which give existence theorems 
for the full water wave equations and as well a proof of 
convergence of solutions in the shallow water scaling limit. 
In both cases the bottom is assumed to be flat, and the authors work
with initial data given in spaces of analytic functions. Results on 
long wave scaling limits of the water waves problem in dispersive
regimes include Craig~\cite{Craig} and Schneider \& Wayne~\cite{SW00}
and their treatment of the two-dimensional problem, a long-time
existence theory, and the Boussinesq and KdV limits in Sobolev 
spaces. More recently, the paper of Lannes~\cite{LJAMS} gives 
an existence theory for solutions of the water wave problem for 
fluid domains with smooth variable bathymetry, and the further paper 
of Alvarez--Samaniego \& Lannes~\cite{AL} gives rigorous results 
on a number of long wave scaling limits of the same problem (see also Iguchi ~\cite{Iguchi} and earlier papers of  Bona et al \cite{BCL} and Chazel\cite{Chazel}),
all papers working with Sobolev space initial data. In the
context of this body of work, what distinguishes the present paper 
is the oscillatory nature of the bottom boundary of the fluid domain,
which has the implications that the solutions themselves are
oscillatory,  and principally, that the
homogenization theory Ansatz giving the form of solutions 
must be justified. Our analysis has several features in common 
with the results of \cite{CSS} on the justification of the nonlinear 
Schr\"odinger equation and the Davey -- Stewartson system as envelope 
equations for modulation theory, the most important of which being 
that the principal theorem is a consistency result rather than a 
full fledged limit theorem for solutions. Nonetheless, as far as we
know this is the first rigorous result which justifies with a 
rigorous analytic argument the application of homogenization 
theory to the water wave problem with rapidly varying periodic 
bathymetry. In the present framework, precise error estimates are needed because  the shallow water limit and the homogenization limit do not commute. More precisely, shallow water expansions  are derived for slowly varying bottoms, neglecting some terms that are relevant for rough bottoms. 
Conversely, homogenization limits are usually performed with  low regularity estimates on solutions,
that place them outside of the regime of high order shallow water asymptotics (see for instance \cite{Chupin} for a recent homogenization result at leading order for the Dirichlet-Neumann operator). The point of our work and the source of many of its technical difficulties is that we perform the homogenization and shallow water limit simultaneously, thereby retaining  the full complement of relevant terms from the original water waves equations.
The (local) effects of this infinity of terms neglected in previous studies add up to create the nonlocal effects present in our approximation.


\subsection{General setting}
The time-dependent fluid domain consists of the fluid domain
$\Omega(b,\zeta)=\{(x,z)\in \R^{d+1},-H_0+b(x)<z<\zeta(x,t)\}$
in which the fluid velocity is represented by the gradient of 
a velocity potential $\Phi$. The dependent variable  $\zeta(x,t)$ 
denotes the surface elevation and $b(x)$ denotes the variation 
of the bottom of the fluid domain from its mean value. We use 
the Hamiltonian formulation due to Zakharov~\cite{Z} and 
Craig \& Sulem~\cite{CS93} in the form of a coupled system 
for the surface elevation $\zeta$ and the trace of the 
velocity potential at the surface
$\psi=\Phi_{\vert_{z=\zeta}}$, namely 
\begin{equation}\label{eq1}	
	\left\lbrace
	\begin{array}{l}
	\dsp \dt\zeta - G[\zeta,b]\psi = 0 ~,   \\
	\dsp \dt \psi+g\zeta
	+\frac{1}{2}\vert\nabla\psi\vert^2
	-\frac{(G[\zeta,b]\psi+\nabla\zeta\cdot\nabla\psi)^2}
	{2(1+\vert\nabla\zeta\vert^2)} = 0 ~.
	\end{array}\right.
\end{equation}
The quantity $G[\zeta,b]\cdot$ is the Dirichlet-Neumann operator, 
defined by
\begin{equation}\label{eq2}
	G[\zeta,b]\psi =
	\sqrt{1+\vert\nabla\zeta\vert^2}\dn\Phi_{\vert_{z=\zeta}} ~,
\end{equation}
where $\Phi$ is the solution of the elliptic boundary value problem
\begin{equation}\label{eq2.5}
	\left\lbrace
	\begin{array}{l}
	\dsp \Delta\Phi+\dz^2\Phi=0 \quad \mbox{ in } \quad \Omega(b,\zeta)~,  \\
	\Phi_{\vert_{z=\zeta}} = \psi,\qquad \dn\Phi_{\vert_{z=-H_0+b }}=0 ~.
	\end{array}\right. 
\end{equation}

Writing the equations of evolution in terms of nondimensional 
variables, different asymptotic regimes of this problem are 
identified by scaling regimes of the associated dimensionless
parameters. Denote by $A$ the typical amplitude of surface waves, 
with $\lambda$ their  typical wavelength. Similarly let $B$ denote the
typical amplitude of the variations of the bottom from its mean value 
$H_0$, with $\ell$ their typical wavelength. From these quantities 
we define the dimensionless variables as follows:
\begin{equation}\label{eq5}
	\begin{array}{llll}
	\dsp x = \lambda X' ~, &
	z = H_0 z' ~, &
	t = \frac{\lambda}{\sqrt{gH_0}}t' ~,  \\
	\dsp \zeta = A\zeta' ~,   &
	\Phi = \frac{A}{H_0}\lambda\sqrt{gH_0}\Phi' ~,  &
	b = B b'(\frac{x}{\ell}) ~.
	\end{array}
\end{equation}
Stemming from this change of variables there are four dimensionless parameters, 
\begin{equation}\label{eq6}
	\mu=\frac{H_0^2}{\lambda^2}~,\qquad \varepsilon=\frac{A}{H_0}~,\qquad 
	\beta=\frac{B}{H_0}~, \qquad \gamma = \frac{\ell}{\lambda} ~.
\end{equation}
Our analysis is concerned with the shallow water regime $\mu\ll
1$. The relative amplitude of solutions is governed by $\varepsilon$. 
In addition to this, the relative amplitude of the bathymetry is given 
by $\beta$, the parameter $\gamma$ determines the relative length of
bottom perturbations with respect to the typical wavelength of surface
waves, and the bottom variations $b'(\cdot)$ are assumed to be
$2\pi$-periodic in all variables. We consider relatively large
amplitude surface waves, meaning that no smallness assumption 
is made on $\varepsilon$.
As usual for this regime, we therefore set $\varepsilon=1$ 
for the sake of simplicity.
With regard to the bottom variations, we set 
\begin{equation}\label{assumparam}
 \beta = \sqrt{\mu} = \gamma \ll 1 ~. 
\end{equation}
The fact that $\beta = \gamma$ corresponds to small bathymetry slope
in this regime, while the {\it roughness strength} is $\rho :=
\sqrt{\mu}/\gamma = 1$. 
For clarity of notation we drop this `prime' notation for the
remainder of the paper.

\subsection{Presentation of  results}

The first result of this paper is the construction of an approximate
solution $(\zeta_a,\psi_a)$ of the water waves problem 
in the form of the Ansatz
\begin{align}\label{Eqn:BasicAnsatz1}
  &  \eps \zeta_a=  \eps \zeta_0 (X,t) + \gamma  \zeta_1(X,X/\gamma,t/\gamma) \\
     \label{Eqn:BasicAnsatz2}
  & \psi_a = \psi_0(X,t) +  \gamma^2 \psi_1(X,X/\gamma,t/\gamma) ~.
\end{align}

\begin{rema}\label{remansatz}
The  factor of $\gamma^2$ in front of the corrector $\psi_1$is natural; indeed, this yields a $O(\gamma)$ corrector for the velocity, which is the physical relevant quantity.
\end{rema}

Setting $V_0 = \nabla\psi_0$ and $h_0 = 1+\eps\zeta_0$, we show that 
$(\zeta_0,V_0)$ satisfies the classical shallow water system with flat bottom,
\begin{equation} \label{shallow-water-eqs}
	\left\lbrace	\begin{array}{l}
    \partial_t \zeta_0 + \nabla\cdot (h_0V_0)=0 ~,\\
    \partial_t V_0 + \nabla\zeta_0 + (V_0\cdot\nabla)V_0 =0 ~,
	\end{array}\right.
\end{equation}
while the corrector terms $(\zeta_1,\psi_1)$ satisfy a linear 
nonlocal coupled system of equations in the fast variables
($\tau= t/\gamma, Y = X/\gamma$)

\begin{equation}\label{shallow-water-corr}	
	\left\lbrace
	\begin{array}{l}
    \dsp \partial_\tau \zeta_1 + V_0\, \cdot \nabla_Y \zeta_1 - 
        |D_Y| \tanh(h_0\,|D_Y|) \psi_1    \\  \qquad \qquad \qquad
          =  V_0\, \cdot \nabla_Y \sech(h_0\,|D_Y|) b ~,   \\
    \dsp \partial_\tau \psi_1 + V_0\, \cdot \nabla_Y\psi_1 +\zeta_1  = 0 ~.
	\end{array}\right.
\end{equation}
In system \eqref{shallow-water-corr}, the functions $\zeta_1, \psi_1$ 
are periodic in the variables $Y$, while the variables $(t,X)$
 are to be 
treated as parameters.
The above system represents the
linearized water wave equations in a fluid region of depth $h_0$, with
a background flow given by the velocity field $V_0$. The source term
of the RHS is due to the effect of scattering of the background flow
from the variable bottom. 

The second result of this paper is a mathematical justification of the
derivation of the above system of model equations 
\eqref{shallow-water-eqs}\eqref{shallow-water-corr}. Our proof is in
the form of a consistency analysis of the Euler equations of free
surface water waves, for which we show that the functions
$(\zeta_a,\psi_a)$ whose constituents satisfy 
\eqref{shallow-water-eqs}\eqref{shallow-water-corr} are approximate
solutions of the Euler equations. They are not in general an exact 
solution, but they satisfy the equations \eqref{eq1} up to an error term
$E_a$, and we show that this error is small. Namely, we prove that 
\[
     | E_a |_{H^*} < C\gamma^{3/8} ~,
\]
where the appropriate norm $|\cdot|_{H^*}$ is defined as 
$ | E_a |_{H^*} =  |E_{a1}|_{L^2} +\gamma^{-3/8} |E_{a2}|_{H^{1/2}}$, and 
$E_a=(E_{a1}, E_{a2})$.  In particular, the error is small for the usual Hamiltonian norm
of the water waves equations.
The most striking point of our analysis is that this result is valid for the natural time scale $t=\BigOh{1}$ associated to (\ref{shallow-water-eqs}) only if the free surface does not resonate with the rapidly varying bottom. Such a resonance is obtained if there exist $(t,X)$ such that
$$
(k\cdot V_0(X,t))^2=|k|\tanh(h_0(X,t)|k|)
$$
for some $k\in\Z$ corresponding to a nonzero mode of the Fourier decomposition of the bottom parametrization $b$.
This condition can be viewed as a nonlinear generalization of the classical Bragg resonance which is obtained when the wavelengths of the free surface and of the bottom are of the same order, while here, the latter is much smaller. In absence of such resonances, it is possible to find {\it locally stationary} solutions for the corrector terms, that is, solutions to (\ref{shallow-water-corr}) that do not depend on the fast time variable $\tau$. 
When such resonances occur, the dependence of the correctors on $\tau$ cannot be removed, and this induces secular growth effects that destroy the accuracy of the approximation (it is only valid on a much smaller time scale, $t=o(1)$, than the relevant one). It is likely that in this case, the dynamics of the leading term $(\zeta_0,V_0)$ is affected, but this point is left for a future study.

The Ansatz \eqref{Eqn:BasicAnsatz1}\eqref{Eqn:BasicAnsatz2} and the 
error estimates for the quantity $E_a$ represent a problem in
homogenization theory. The principal terms $(\zeta_0,\psi_0)$ are 
solutions of an effective equation, and the multiscale terms 
$(\zeta_1,\psi_1)$ are the first corrector terms. The dynamics of 
the Euler equations require solving an elliptic equation at each
instant of time, on an unknown domain $\Omega(b_\gamma,\zeta)$ whose
boundaries are defined by oscillatory functions. The approach we 
take in this paper to the analysis of this elliptic problem and its
asymptotic behavior is to transform this domain to a reference 
domain $\Omega_0$, resulting in an elliptic problem with rapidly
varying periodic coefficients. The principal (effective) term and 
the correctors are derived from this problem, with the principal term
solving an effective equation, and the corrector solving an
appropriate cell problem. These are then used to express the 
Dirichlet -- Neumann operator on the free surface of the fluid
domain, which in turn is used to express the evolution equations 
\eqref{eq1}. The dynamics of the short spatial scales are separated 
from the evolution of the long scales using the concept of convergence
on two scales \cite{A}. The principal part of our mathematical
analysis is to control the error estimates of the homogenization
approximation \eqref{Eqn:BasicAnsatz1}\eqref{Eqn:BasicAnsatz2} in
describing the solutions of this elliptic boundary value problem and
the associated expression for the Dirichlet -- Neumann problem.


\section{Euler Equations}
Zakharov showed that
the water wave problem can be written in  
the Hamiltonian form \cite{Z}

\begin{equation}\label{eq3}
	\dt\left(\begin{array}{c}
	\zeta\\
	\psi
	\end{array}\right)
	=\left(\begin{array}{cc} 
	0 & I\\ 
	-I & 0\end{array}\right)
	\left(\begin{array}{c}
	\delta_\zeta H\\ 
	\delta_\psi H\end{array}\right),
\end{equation}
where the canonical variables are the surface elevation $\zeta$ and the trace of the velocity
potential on the free surface $\psi=\Phi_{\vert_{z=\zeta}}$,
and the Hamiltonian $H$ is given by
\begin{equation}\label{eq4}
	H(\zeta,\psi)=\frac{1}{2}\int_{\R^d}\psi G[\zeta,b]\psi + g
        \zeta^2 \, dX ~.
\end{equation}
The system for ($\zeta,\psi)$ is written as  \eqref{eq1}, which in
   dimensionless form  becomes
\begin{equation}\label{eq7}
	\left\lbrace
	\begin{array}{l}
	\dsp \dt \zeta-\frac{1}{\mu}
	G_{\mu}[\eps\zeta,\beta b_\gamma]
	\psi=0 ~,\\
	\dsp \dt \psi+\zeta+\frac{1}{2}\vert\nabla\psi\vert^2
	-\eps\mu\frac{(\frac{1}{\mu}G_{\mu}
	[\eps\zeta,\beta b_\gamma]\psi
	+\eps\nabla\zeta\cdot\nabla\psi)^2}
	{2(1+\mu\vert\nabla\zeta\vert^2)}=0 ~,
	\end{array}\right.
\end{equation}
where $b_\gamma(\cdot)=b(\cdot/\gamma)$ and where 
$G_{\mu}[\eps\zeta,\beta b_\gamma]$ is the nondimensionlized  
Dirichlet -- Neumann operator defined by
\begin{equation}\label{eq8}
	G_{\mu}[\eps\zeta,\beta b_\gamma]\psi=
	\sqrt{1+\vert\nabla\zeta\vert^2}\dn\Phi_{\vert_{z=\eps\zeta}}
\end{equation}
and where $\Phi$ is the potential function, satisfying
\begin{equation}\label{eq9}
	\left\lbrace
	\begin{array}{l}
	\dsp \mu \Delta\Phi+\dz^2\Phi=0\mbox{ in } \Omega,\\
	\Phi_{\vert_{z=\eps\zeta}}=\psi,\qquad 
\dn\Phi_{\vert_{z=-1+\beta b_\gamma}}=0,
	\end{array}\right.
\end{equation}
in the  fluid domain $\Omega(b_\gamma,\zeta)$, 
\[
	\Omega(b_\gamma,\zeta) = \{(X,z)\in \R^{d+1},
	-1+\beta b(X/\gamma)<z<\eps\zeta(X)\} ~.
\] 
The operator $\dn$ is the outwards \emph{conormal} derivative associated
with the operator $\mu\Delta+\dz^2$. One can rewrite \eqref{eq7} in 
Hamiltonian form \eqref{eq3}, replacing the Hamiltonian $H$ given by 
\eqref{eq4} by its nondimensional form
\begin{equation}\label{eq10}
	H(\zeta,\psi) = 
	\frac{1}{2}\int_{\R^d} \bigl(
	\psi \frac{1}{\mu}G_\mu[\eps\zeta,\beta  b_\gamma]\psi
	+ \zeta^2 \bigr) \, dX ~.
\end{equation}

\subsection{Notation}
We denote by $d=1$ or $2$ the horizontal dimension of the fluid
domain, and by $X\in \R^d$ the horizontal variables, while $z$ is 
the vertical variable. We denote by ${\bf e_z}$ the unit upward 
vertical vector. 

The domain and the potential function  will depend upon
both  regular and  rapidly oscillating variables, which we denote 
$X\in \R^d$ and $Y\in \T^d =\R^d/(2\pi\Z)^d$, respectively. That is, 
we will give data for the water wave problem which is of a multiscale 
nature, with the fixed multiscale bottom variations as well, and we 
will seek solutions which have a well defined asymptotic expansion 
in terms of multiscale quantities.
 To express this, we use the classical  notation of 
a multiscale function  that is, a function $f(X,Y)$ 
defined on $\R^d\times \T^d $, for which the realization is the trace
$f\trace=f(X,X/\gamma)$, \cite{BLP}.  In the problem we consider, there are other
variables as well, such as the vertical variable $z \in [-1,0]$, for which  $f=f(X,Y,z)$ is a 
multiscale function whose realization is 
 $f(X,X/\gamma,z)$.

The differential operators $\nabla$ and $\Delta$ act on functions of the horizontal 
variable $X$. The operator $\Lambda$ is defined by 
$\Lambda := (1 - \Delta)^{1/2}$. We use the standard notation for 
Fourier multipliers, namely  $D=\frac{1}{i}\nabla$ and 
$\widehat{f(D)u}(k)=f(k)\widehat{u}(k)$.  When applied to multiscale
functions, we distinguish this fact using the  notation $\nabla_Y$, $\Delta_Y$, $D_Y$, 
when  differential operators act specifically on the fast
variables $Y$, and $\nabla_X$, $\Delta_X$, $D_X$ when they act on 
the long scale $X$ variables. Finally, the notation $\nam$ stands for 
$\nam=(\sqrt{\mu}\nabla^T, \dz)^T$. 

We encounter functions defined on the fluid domain $\Omega(b_\gamma,\zeta)$
or the reference domain $\Omega_0= \R^d\times(0,1)$, as well as functions defined on
the free surface, parametrized by $X\in \R^d$. The notation used for
function space norms is that $\| \cdot \|_{L^2}$, $\| \cdot \|_{H^r}$
 is used for the classical Sobolev space norms over $\Omega_0$,
while for norms defined over the boundary $X\in \R^d$ we use the
notation $| \cdot |_{L^2}$, $| \cdot |_{H^r}$. Norms of multiscale functions are given 
similarly, for example $~|\cdot |_{L^2(C^1_Y)}$.

For all $r_1,r_2 \geq 0$, we also define the space $H^{r_1,r_2}=H^{r_1,r_2}(\R^d\times \T^d)$ by
\begin{equation}\label{defHrr}
H^{r_1,r_2}(\R^d\times \T^d)=\{f\in L^2(\R^d\times \T^d), \abs{f}_{H^{r_1,r_2}}<\infty\},
\end{equation}
with
$\abs{f}_{H^{r_1,r_2}}^2=\abs{(1-\Delta_X)^{r_1/2}(1-\Delta_Y)^{r_2/2}f}_{L^2(\R^d\times \T^d)}^2$.

\subsection{Change of variables and domain}
The first component of the Hamiltonian (\ref{eq10}) corresponds
to the nondimensionalized kinetic energy. It follows from the
definition of $G_\mu[\eps\zeta,\beta b_\gamma]$ and Green's identity that
\begin{equation}\label{eq11}
	\int_{\R^d}
	\psi \frac{1}{\mu}G_\mu[\eps\zeta,\beta b_\gamma]\psi \, dX =
	\frac{1}{\mu}\int_\Omega \vert \nam \Phi\vert^2 \, dzdX ~,
\end{equation}
where $\Phi$ is the velocity potential (\ref{eq9}). Since this 
expression depends on   $\beta$ and $\gamma$ through the 
domain of integration $\Omega(b_\gamma,\zeta)$, it is convenient 
to transform it into an integral over a fixed domain independent 
of the parameters and of the perturbations $\zeta$ and $b$. Under 
the assumption that the fluid height $h = 1+\eps\zeta-\beta b_\gamma$
 is always non-negative, namely
\begin{equation}\label{eq12}
	\exists \alpha>0,\qquad 1+\eps\zeta-\beta b_\gamma
	\geq \alpha\mbox{ on }\R^d ~,
\end{equation}
an explicit diffeomorphism $S$ mapping the flat strip 
$\Omega_0$ onto the fluid domain $\Omega$ is given by
\begin{equation}\label{eq13}
	S:
	\begin{array}{ccc}
	{ \Omega_0}&\to &\Omega\\
	(X,z)&\mapsto & \big(X,z+\sigma(X,z)\big),
	\end{array}
\end{equation}
where $\sigma(X,z)=\eps (z+1)\zeta(X) -z\beta b_\gamma(X)$. 
We have in particular $h =1 +\partial_z \sigma$. 

Defining $\phi$ on 
$\Omega_0$ by $\phi=\Phi\circ S$, 
one can check (see Prop. 2.7 of \cite{LJAMS} and \S 2.2 of \cite{AL})
that the new potential function $\phi$ solves
\begin{equation}\label{eq14}
	\left\lbrace
	\begin{array}{l}
	\nam\cdot P[\sigma]\nam\phi=0,\\
	\phi_{\vert_{z=0}}=\psi,\qquad \dn\phi_{\vert_{z=-1}}=0,
	\end{array}\right.
\end{equation}
where $\dn\phi_{\vert_{z=-1}}$ is the outward conormal derivative
in the new variables
\[
	\dn\phi_{\vert_{z=-1}}=
	-{\bf e_z}\cdot P[\sigma]\nam\phi_{\vert_{z=-1}}
\]
and where the matrix $P[\sigma]$ is given by
\begin{equation}\label{eq15}
	P[\sigma]=\left(\begin{array}{cc}
	h I & -\sqrt{\mu} \nabla \sigma\\
	-\sqrt{\mu} \nabla \sigma^T&
	\dsp \frac{1 +\mu\vert \nabla \sigma\vert^2}{h}
	\end{array}\right),
   \quad\mbox{ with }\quad h=1+\eps\zeta-\beta b_\gamma.
\end{equation}




\section{Multiple scale asymptotic expansions}

\subsection{Ansatz and decomposition of the solutions}
This section is devoted to the study of the elliptic problem \eqref{eq14}
where $(\zeta,\psi)$ are given; the time is fixed and appears as a parameter. 
We pose the multiple-scale Ansatz on $(\zeta,\psi)$ :
\begin{equation}
 \label{ansatz}
   \begin{array}{l}
    \eps \zeta = \eps \zeta_0(X) + \gamma \zeta_1(X,X/\gamma) \\
    \psi = \psi_0(X) + \gamma^2\psi_1(X,X/\gamma).
   \end{array}
\end{equation}
Recalling that  $\beta=\gamma=\sqrt{\mu}$, this leads to the decomposition of the height function of the fluid domain
\[
   h = h_0 +\beta h_1 ~, \quad\mbox{ where }\quad  
   h_0 = 1 +\eps \zeta_0 \quad\mbox{ and }\quad h_1 = \zeta_1-b_\gamma ~.
\]
Similarly, the new vertical deformations are posed in terms of this Ansatz
\[
   \sigma =  \sigma_0 +\beta\sigma_1 ~, \quad \mbox{ where }\quad 
   \sigma_0 = (z+1) \eps \zeta_0 \quad\mbox{ and }\quad 
   \sigma_1  =(z+1)  \zeta_1 - z b_\gamma ~.
\]
The coefficients $P[\sigma]$ are then written as
\[
    P[\sigma] = P_0+ \beta P_1,
\]
with 
\begin{equation}\label{eq16}
	P_0=P[\sigma_0] \quad 
	\mbox{ and } \quad 
	 \beta P_1 =P[\sigma] - P_0 ~.
\end{equation}
Explicitly
\[
	P_1 = \left(\begin{array}{cc}
	(\zeta_1 -b) I  & -\sqrt{\mu} \nabla \sigma_1\\
	-\sqrt{\mu} \nabla \sigma_1^T &
      \frac{1}{\beta} \Big( \frac{ 1 + \mu |\nabla \sigma|^2}{h}-
      \frac{1 +\mu|\nabla \sigma_0|^2}{h_0} \Big)
	\end{array}\right) ~.
\]
We accordingly  decompose the potential function $\phi$ as
\begin{equation}
  \phi = \phi_0(X,z) +  \beta \gamma \chi(X,z; \gamma) =
  \phi_0(X,z) +  \mu \chi(X,z; \gamma)
\end{equation}
where all the contributions coming from the roughness are contained 
in $\chi$. This section is devoted to deriving asymptotic expansions, 
with accompanying error estimates on the two components $\phi_0$ and
$\chi$, in the limit $\mu \to 0$. In order to do so, we must augment
\eqref{eq12} with the assumption that 
\begin{equation}\label{eq12bis}
    \exists \alpha_0>0,\qquad 1+\eps\zeta_0
    \geq \alpha_0\mbox{ on }\R^d ~.
\end{equation}
This ensures that the water depth does not vanish for the
averaged fluid domain that arises when all the fluctuations 
due to the roughness are neglected. Assumption \eqref{eq12bis} ensures 
the coercivity of $P_0$.

\begin{prop}\label{lemm2}
Let $\zeta,b \in W^{1,\infty}(\R^d)$ and assume that \eqref{eq12} and \eqref{eq12bis}
are satisfied. Then for all $\psi$ such that $\nabla\psi \in H^{1/2}(\R^d)^d$, 
there exists a unique solution $\phi$ to \eqref{eq14} such that 
$\nam\phi\in H^1(\Omega_0)^{d+1}$. Moreover, $\phi_0$ and $\chi$ solve 
\begin{equation}\label{eq17}
	\left\lbrace
	\begin{array}{l}
	\nam\cdot P_0\nam\phi_0=0,\\
	\phi_0\,_{\vert_{z=0}}=\psi_0,\qquad 
	-{\bf e_z}\cdot  P_0\nam \phi_0\,_{\vert_{z=-1}}=0,
	\end{array}\right.
\end{equation}
and
\begin{equation}\label{eq18}
	\left\lbrace
	\begin{array}{l}
	\nam\cdot P[\sigma]\nam\chi=-\frac{1}{\gamma}
         \nam\cdot P_1\nam\phi_0,\\
	\chi\,_{\vert_{z=0}}=\psi_1,\qquad 
	-{\bf e_z}\cdot P[\sigma]\nam\chi\,_{\vert_{z=-1}}=\frac{1}{\gamma}
		{\bf e_z}\cdot P_1\nam\phi_0\,_\bott.
	\end{array}\right.
\end{equation}
\end{prop}

\begin{proof}
The existence of a unique solution $\phi$ such that 
$\nam\phi\in H^1(\Omega_0)^{d+1}$ to \eqref{eq14} is a
classical result, and we thus omit the proof. Similarly, 
there exists a unique solution $\phi_0$ such that 
$\nam\phi_0\in H^1(\Omega_0)^{d+1}$ to \eqref{eq17} since the
boundary condition on the lower boundary is the conormal 
derivative associated to the elliptic operator 
$\nam\cdot P_0\nam$. It remains to prove that $\chi$ 
solves \eqref{eq18}. A calculation gives that
\begin{eqnarray*}
	\lefteqn{\dn\phi_{\vert_{z=-1}}
	:=- {\bf e_z}\cdot P[\sigma]\nam\phi_{\vert_{z=-1}}}\\
 	&=&- \beta\gamma{\bf e_z}\cdot P[\sigma]\nam\chi\,_{\vert_{z=-1}}
	-{\bf e_z}\cdot P_0\nam\phi_0\,_{\vert_{z=-1}}
	- \beta {\bf e_z}\cdot P_1\nam\phi_0\,_\bott ~.
\end{eqnarray*}
Since by assumption one also has $\dn\phi_{\vert_{z=-1}}=0$ and 
$-{\bf e_z}\cdot P_0\nam\phi_0\,_{\vert_{z=-1}}=0$,
one has
\[
	-{\bf e_z}\cdot P[\sigma]\nam\chi\,_{\vert_{z=-1}}= \frac{1}{\gamma}
	{\bf e_z}\cdot P_1\nam\phi_0\,_\bott.
\]
It is straightforward to check that $\chi\,_{\vert_{z=0}}=\psi_1$ and that
\begin{eqnarray*}
	\nam\cdot P[\sigma]\nam\chi
	&=&\frac{1}{\gamma\beta}(\nam\cdot P[\sigma]\nam\phi-	
	\nam\cdot P[\sigma] \nam\phi_0)\\
	&=& -\frac{1}{\gamma}	\nam\cdot P_1\nam\phi_0,
\end{eqnarray*}
and the result follows.
\end{proof}



\subsection {Asymptotic analysis with estimates of $\nam \phi_0$}
In this section we prove an estimate on $\nam\phi_0$, and we give the
first terms of its asymptotic expansion in the limit as $\mu\to
0$. For purposes of understanding the $H^{-1/2}$-norm of the trace 
of $\nam\phi_0$ on the 
free surface $\{z=0 \}$,we use $L^2$ estimates on $\Omega_0$ of both $\nam\phi_0$ 
and its generalized Riesz transform, given by 
$\Lambda^{-1}\partial_z\nam\phi_0$. This is generalized to higher order norms.

\begin{prop}\label{lemm3}
	Let $r\in {\mathbb N}$ and $\zeta_0 \in W^{1+r,\infty}\cap W^{2,\infty}(\R^d)$ and 
	assume that (\ref{eq12bis}) is
	satisfied
	for some $\alpha_0>0$. Then \\
	{\bf (i.)} For all $\mu\in (0,1)$ and all
	$\psi_0$ such that $\nabla\psi_0\in H^r(\R^d)^d$, 
	the solution $\phi_0$ to \eqref{eq17} satisfies
\begin{eqnarray*}
   && \Vert \Lambda^r\nam \phi_0\Vert_{L^2} \leq \sqrt{\mu} 
	    C(\frac{1}{\alpha_0},\vert\zeta_0\vert_{W^{1+r,\infty}})
	    \vert \nabla\psi_0\vert_{H^r} ~,  \\
   && \qquad \Vert \Lambda^{r-1}\partial_z\nam\phi_0\Vert_{L^2} \leq 
            \mu C(\frac{1}{\alpha_0},\vert\zeta_0\vert_{W^{1+r,\infty}},\vert \zeta_0\vert_{W^{2,\infty}}) 
            \vert \nabla\psi_0\vert_{H^r} ~.  
\end{eqnarray*}
	{\bf (ii.)}  If  $\nabla\psi_0\in H^{r+1}(\R^d)^d$, 
	one also has
\begin{eqnarray*}
  && \Vert \Lambda^r(\nam \phi_0-\nam\psi_0)\Vert_{L^2} \leq \mu 
	    C(\frac{1}{\alpha_0},\vert\zeta_0\vert_{W^{1+r,\infty}})
            \vert \nabla\psi_0\vert_{H^{r+1}} ~;   \\
  &&  \qquad \Vert\Lambda^{r-1} \partial_z( \nam\phi_0 - \nam\psi_0)\Vert_{L^2} 
          \leq \mu C(\frac{1}{\alpha_0},\vert\zeta_0\vert_{W^{1+r,\infty}},\vert \zeta_0\vert_{W^{2,\infty}})
            \vert \nabla\psi_0\vert_{H^{r+1}}.
\end{eqnarray*}
	{\bf (iii.)}  Suppose that $\zeta_0 \in W^{2+r,\infty}(\R^d)$
        and $\Delta\psi_0\in H^{2+r}(\R^d)$, and set 
\[
	\phi_0^{(1)}=- h_0^2(\frac{z^2}{2}+z)\Delta\psi_0 ~, \qquad 
	( h_0:=1+\eps\zeta_0) ~.
\]
as the next term of the asymptotic expansion. Then there are estimates 
of the remainder, in the form
\begin{eqnarray*}
   &&   \Vert \Lambda^r\big(\nam \phi_0-\nam (\psi_0-\mu \phi_0^{(1)})\big)
           \Vert_{L^2}\leq \mu^2
	   C(\frac{1}{\alpha_0},\vert\zeta_0\vert_{W^{2+r,\infty}})
           \vert \Delta\psi_0\vert_{H^{2+r}} ~;  \\
   && \qquad \Abs{\Lambda^{r-1}\partial_z \nam (\phi_0 - \psi_0 - \mu\phi_0^{(1)})}_{L^2}
       \leq \mu^2
	   C(\frac{1}{\alpha_0},\vert\zeta_0\vert_{W^{2+r,\infty}})
           \vert \Delta\psi_0\vert_{H^{2+r}} ~.  
\end{eqnarray*}
\end{prop}

\begin{proof}The first inequality of the proposition is obtained by  standard elliptic 
estimates for $\Vert \nam \phi_0\Vert_{L^2}$ (see Corollary 2.2 of \cite{AL}). The Riesz transform of 
$\nam\phi_0$ has components 
$\Lambda^{-1} \partial_z \nam\phi_0 
   = (\sqrt{\mu} \Lambda^{-1} \nabla\partial_z \phi_0, 
     \Lambda^{-1} \partial_z^2\phi_0)$;
estimates of the first component come from the first inequality 
of {\bf (i)}, since $\Lambda^{-1} \nabla$ is $L^2$ bounded. 
Estimates of the second component are obtained through the 
equation \eqref{eq17} itself. Namely one has an expression for 
$\partial_z^2\phi_0$ in the form
\begin{eqnarray}
\label{madrid}
  && \dz^2\phi_0=\mu\frac{h_0}{1+\mu\vert\nabla\sigma_0\vert^2}\Big[-\frac{\dz\vert\nabla\sigma_0\vert^2}{h_0}\dz\phi_0 + 
      \partial_z (\nabla\sigma_0 \cdot \nabla \phi_0)
     + \nabla\cdot( \nabla\sigma_0 \partial_z \phi_0 )
     - \nabla\cdot( h_0 \nabla \phi_0) \Big] ~.
\end{eqnarray}
In order to get an estimate on $\Vert\Lambda^{r-1}\dz^2\phi_0\Vert_2$, we need the following lemma.
\begin{lemm}\label{lemmfrac}
Let $r\in {\mathbb N}$ and $F$ and $G\geq 0$ be such that $\Lambda^{r-1} F\in L^2(\Omega_0)$
and $G\in L^\infty((-1,0);W^{\vert r-1\vert,\infty}(\R^d))$. Then
$$
\Vert \Lambda^{r-1} (\frac{F}{1+G})\Vert_{L^2}\leq C(\Vert G\Vert_{L^\infty_z W^{\vert r-1\vert,\infty}_X})\Vert \Lambda^{r-1} F\Vert_{L^2}.
$$
\end{lemm}
\begin{proof}
Just write 
$
\frac{F}{1+G}=F-F\frac{G}{1+G}.
$
Recalling that one has
$\abs{fg}_{H^{r-1}}\lesssim \abs{f}_{H^{r-1}}\abs{g}_{W^{\vert r-1\vert,\infty}}$, we 
get
\begin{eqnarray*}
\Vert \Lambda^{r-1} (\frac{F}{1+G})\Vert_{L^2}&\lesssim& \Vert \Lambda^{r-1} F\Vert_{L^2}
\big(1+\Vert{\frac{G}{1+G}\Vert}_{L^\infty_z W^{\vert r-1\vert,\infty}_X}\big)\\
&\lesssim& C(\Vert G\Vert_{L^\infty_z W^{\vert r-1\vert,\infty}_X})\Vert \Lambda^{r-1} F\Vert_{L^2}.
\end{eqnarray*}
\end{proof}
Applying this lemma to (\ref{madrid}) with $G=\mu\vert \nabla\sigma_0\vert^2$, one easily gets
$$
\Vert \Lambda^{r-1}\dz^2\phi_0\Vert_{L^2}\leq \mu C(\frac{1}{\alpha_0},
\vert \zeta_0\vert_{W^{r+1,\infty}})
\frac{1}{\sqrt{\mu}}\Vert \Lambda^r\nam\phi_0\Vert_{L^2},
$$
and the estimate follows from the control on $\Vert \Lambda^r \nam\phi_0\Vert_{L^2}$ 
established above.

For the second point of the proposition, we write $\phi_0 = \psi_0 +\mu \chi_0^{(1)}$. 
The resulting system for $\chi_0^{(1)} $ is
\begin{equation}\label{Eqn:SystemForChi0}
	\left\lbrace
	\begin{array}{l}
	\nam\cdot P_0\nam \chi_0^{(1)} = - \frac{1}{\mu}
\nam\cdot P_0\nam \psi_0 \\
\chi_0^{(1)}\,_{\vert_{z=0}}=0,\qquad 
	-{\bf e_z}\cdot P_0\nam \chi_0^{(1)}\,_{\vert_{z=-1}}=\frac{1}{\mu}
       {\bf e_z}\cdot P_0\nam \psi_0\,_{\vert_{z=-1}}.
	\end{array}\right.
\end{equation}
A calculation shows that
\[
 - \frac{1}{\mu}
\nam\cdot P_0\nam \psi_0 =- \nabla\cdot (h_0\nabla\psi_0 )+
 \partial_z (\nabla\sigma_0\cdot \nabla \psi_0) =- h_0 \Delta\psi_0
\]
and that ${\bf e_z}\cdot P_0\nam \psi_0\,_{\vert_{z=-1}} =0$, we obtain 
	\begin{equation}\label{eq2.14}
	\left\lbrace
	\begin{array}{l}
	\nam\cdot P_0\nam \chi_0^{(1)} =- h_0 \Delta\psi_0,\\
\chi_0^{(1)}\,_{\vert_{z=0}}=0,\qquad 
	-{\bf e_z}\cdot P_0\nam \chi_0^{(1)}\,_{\vert_{z=-1}}=0.
	\end{array}\right.
	\end{equation}
Multiplying the equation by $\chi_0^{(1)}$ and integrating by parts, we get
\[
	\int_{\Omega_0}\nam \chi_0^{(1)}\cdot P_0\nam \chi_0^{(1)} \,
        dzdX  = \int_{\Omega_0} \chi_0^{(1)}  h_0 \Delta\psi_0 \, dzdX
        ~.
\]
Using the coercivity of the matrix $P_0$ (Prop. 2.3.iii of \cite{AL}), the
Cauchy-Schwarz inequality and Poincar\'e inequality (in order to control
$\Vert \chi_0^{(1)}\Vert_{L^2}$ by $\Vert \nam \chi_0^{(1)}\Vert_{L^2}$), 
one gets
\[
  \Vert \nam \chi_0^{(1)}\Vert_{L^2}\leq  C(\frac{1}{\alpha_0},
  \vert\zeta_0\vert_{W^{1,\infty}})| \Delta \psi_0|_{L^2}
\]
and the result follows. Higher order estimates are handled similarly 
and only require the control of additional commutator estimates. We 
omit these classical details. The estimate of the generalized Riesz
transform $\Lambda^{-1}\partial_z\nam(\phi_0 - \psi_0)$ is similar to the
analog estimate in {\bf (i)} of this proposition.

For the third point of the proposition, we solve  \eqref{eq2.14} 
at lowest order in $\mu$. We 
write $\chi_0^{(1)} = \phi_0^{(1)} +\mu \chi_0^{(2)} 
$, or equivalently 
\[
    \phi_0 = \psi_0 +    \mu \phi_0^{(1)} +\mu^2 \chi_0^{(2)}
\]
with $\phi_0^{(1)} =- h_0^2(\frac{z^2}{2}+z)\Delta\psi_0$. The 
correction $\chi_0^{(2)}$ satisfies the system
\begin{equation}\label{eq2.15}
	\left\lbrace
	\begin{array}{l}
	\dsp \nam\cdot P_0\nam \chi_0^{(2)} =-
 \nabla \cdot(h_0 \nabla \phi_0^{(1)} -\nabla  \sigma_0 \partial_z \phi_0^{(1)})\\
\qquad\qquad\qquad\quad\dsp + \partial_z( \nabla \sigma_0 \cdot \nabla \phi_0^{(1)} -
\frac{|\nabla \sigma_0|^2}{h_0} \partial_z \phi_0^{(1)}).
  \\
	\dsp \chi_0^{(2)}\,_{\vert_{z=0}}=0,\qquad 
	-{\bf e_z}\cdot P_0\nam \chi_0^{(2)}\,_{\vert_{z=-1}}=0,
	\end{array}\right.
	\end{equation}
where we used the fact that ${\bf e_z}\cdot P_0\nam \phi_0^{(1)}\,_{\vert_{z=-1}} = 0$ 
to obtain the bottom boundary conditions. Proceeding as above, we get the result.
\end{proof}

\subsection{Asymptotic analysis with estimates of $\chi$}

To find an asymptotic expansion of $\chi$,
our starting point is the equation \eqref{eq18}
for $\chi$. Decompose the solution as the sum of a multiscale function 
and a correction term,
\begin{equation}\label{exp-chi}
   \chi = \phi_1^{(0)}(X,Y,z)\trace
        + \sqrt{\mu}\chi_1^{(1)}(X,z;\gamma) ~.
\end{equation}
When acting on a multiscale function of the variables $(X,X/\gamma)$, 
the operator $\nam$ becomes $\nabla_{Y,z} + \nam_X$, where 
$\nam_X = \left( \begin{matrix} \sqrt{\mu}\nabla_X\\ 0\end{matrix}\right)$:
\[
   \nam (f(X,Y)\trace)=\big[(\nabla_{Y,z}+\nam_X)f(X,Y)\big]\trace.
\]
We can therefore write
\begin{eqnarray*}
   \nam\cdot P[\sigma]\nam\chi & = & (\nabla_{Y,z}+\nam_X)
    \cdot P_0(\nabla_{Y,z}+\nam_X)\phi_1^{(0)}\trace    \\
  & & + \beta \nam \cdot P_1\nam \phi_1^{(0)}
      +\sqrt{\mu}\nam\cdot P[\sigma]\nam \chi_1^{(1)},
\end{eqnarray*}
so that (\ref{eq18}) becomes (recall that $\beta=\sqrt{\mu}$)

\begin{equation}\label{eq2.16bis}
   \left\lbrace
  \begin{array}{l}
   \dsp \nabla_{Y,z}\cdot P_0\nabla_{Y,z}\phi_1^{(0)}\trace 
    + \sqrt{\mu}\nam\cdot P[\sigma]\nam \chi_1^{(1)}  \\
   \hspace{1cm}\dsp =
    -\frac{1}{\gamma}\nam\cdot P_1\nam\phi_0 
    + \sqrt{\mu}\nam \cdot \widetilde{A}+\sqrt{\mu}\widetilde{g} ~,  \\
      ({\phi_1^{(0)}}_\trace+\sqrt{\mu}\chi_1^{(1)}) \,_{\vert_{z=0}} = {\psi_1}_\trace ~,  \\
   \dsp 	
     \big(-{\bf e_z}\cdot P_0\nabla_{Y,z}\phi_1^{(0)}\trace\big)
     \,_{\vert_{z=-1}}-\sqrt{\mu}{\bf e_z}\cdot P[\sigma]\nam\chi_1^{(1)}\,_{\vert_{z=-1}}\\
    \hspace{1cm}\dsp =\frac{1}{\gamma}
     {\bf e_z}\cdot P_1\nam\phi_0\,_\bott - \sqrt{\mu}{\bf e_z} 
     \cdot \widetilde{A}\,_{\vert_{z=-1}}  
  \end{array} \right.
\end{equation}
with\footnote{The operator $\nam$ always acts on multiscale functions
  on the two variables $X$ and $z$ (and not on $Y$). The notation 
  $\nam \phi_1^{(0)}$ is therefore a shortcut for $\nam(\phi_1^{(0)}\trace)$.}
\[
   \widetilde{A}=-P_1\nam\phi_1^{(0)}-P_0\left(\begin{array}{c}\nabla_X\phi_1^
{(0)}\\0\end{array}\right)_\trace
\]
and
\[
  \widetilde{g}= -\begin{pmatrix} \nabla_X \\ 0 \end{pmatrix} 
  \cdot ( P_0 \nabla_{Y,z} \phi_1^{(0)})_\trace ~.
\]

In order to make the leading order terms in (\ref{eq2.16bis}) explicit, 
we further decompose $P_0$ as $P_0= P_0^{(0)}+\sqrt{\mu} P_0^{(1)}$ 
and $P_1 = P_1^{(0)} + \sqrt{\mu} P_1^{(1)}$, with
\[
  P_0^{(0)} = \left(\begin{matrix} h_0 I&0\\ 0&\frac{1}{h_0} I
    \end{matrix}\right), \qquad P_0^{(1)} 
   = \left(\begin{matrix} 0 & - \nabla\sigma_0  \\
   - \nabla\sigma_0^T & \frac{\sqrt{\mu} |\nabla\sigma_0|^2}{h_0}
     \end{matrix}\right) 
\]
and
\[ 
   P_1^{(0)} = \left(\begin{matrix} (\zeta_1-b) I  &-\nabla_Y \sigma_1  \\ 
  -\nabla_Y  \sigma_1^T & \frac{b-\zeta_1}{h_0^2} I
     \end{matrix}\right), \qquad 
     P_1^{(1)}=\left(\begin{matrix} 0 & - \nabla_X\sigma_1  \\
  - \nabla_X\sigma_1^T & p_{22}^{(1)}
    \end{matrix}\right) 
\]
where the $(2,2)$-coefficient of $P_1^{(1)}$ is 
\[  
   p_{22}^{(1)} = \mu^{-1/2}\Big(\beta^{-1}
     ( \frac{ 1 + \mu |\nabla \sigma|^2}{h} - \frac{1 +\mu|\nabla \sigma_0|^2}{h_0}) 
      -\frac{b-\zeta_1}{h_0^2}\Big) ~.
\]

\begin{lemm}\label{lemmeP1}
The coefficient matrix $P_1 = P_1(\sigma)$ has multiscale functions as coefficients. 
Considered as $P_1 = P_1(X,X/\sqrt{\mu})$, the following estimates hold 
for all $r\in {\mathbb N}$:
\begin{equation*}
   \Abs{P_1^{(0)}}_{L^\infty_z W^{r,\infty}} + \Abs{\partial_z P_1^{(0)}}_{L^\infty_z W^{r,\infty}}  \leq  \mu^{-r/2}
C(\frac{1}{\alpha_0}, \abs{\zeta_0}_{C^{r+1}})
(\abs{\zeta_1}_{C^r} +\abs{\nabla_Y\zeta_1}_{C^r}+\abs{b}_{C^{r+1}}),
   \end{equation*} 
   
\begin{eqnarray*}
&& \Abs{P_1^{(1)}}_{L^\infty_z W^{r,\infty}} + \Abs{\partial_z P_1^{(1)}}_{L^\infty_z W^{r,\infty}}  \leq  \mu^{-r/2}
    C(\frac{1}{\alpha_0},\frac{1}{\alpha}, \abs{\zeta_0}_{C^{r+1}},
\abs{\zeta_1}_{C^r},|\nabla_X\zeta_1|_{C^r}, \abs{\nabla_Y\zeta_1}_{C^r}, \abs{b}_{C^{r+1}}) ~. 
\end{eqnarray*}

\end{lemm}

\begin{proof}
The proof follows by inspecting the elements of $P_1$, $P_1^{(0)}$ and $P_1^{(1)}$.
\end{proof}

Given the above decompositions of $P_0$ and $P_1$, the first term of
the LHS of \eqref{eq2.16bis} is
\[
    \nabla_{Y,z}\cdot P_0\nabla_{Y,z}\phi_1^{(0)} 
  = \nabla_{Y,z}\cdot P_0^{(0)}\nabla_{Y,z}\phi_1^{(0)}
    + \sqrt{\mu}\nabla_{Y,z}\cdot P_0^{(1)}\nabla_{Y,z}\phi_1^{(0)}
\]
and, using that $\gamma=\sqrt{\mu}$, the first term of the RHS of 
\eqref{eq2.16bis} is 
\begin{eqnarray*}
   \frac{1}{\gamma}\nam\cdot P_1\nam\phi_0 
    & = & \frac{1}{\gamma}\nam\cdot P_1^{(0)}\nam \psi_0
      + \frac{1}{\gamma}\nam\cdot P_1^{(0)}\nam(\phi_0-\psi_0) 
      + \nam\cdot P_1^{(1)}\nam\phi_0 \\
    & = & \nabla_{Y,z}\cdot P_1^{(0)}\left( \begin{matrix} \nabla\psi_0
       \\
   0\end{matrix}\right)\trace 
  + \nam_X\cdot P_1^{(0)}\left( \begin{matrix} \nabla\psi_0  \\
    0\end{matrix}\right)\trace \\
    & & + \frac{1}{\gamma}\nam\cdot P_1^{(0)}\nam(\phi_0-\psi_0) 
  + \nam\cdot P_1^{(1)}\nam\phi_0  ~.
\end{eqnarray*}

Extracting the principal term from these two expressions, we deduce that

\begin{equation}\label{eq2.16ter}
  \left\lbrace
  \begin{array}{l}
   \dsp \nabla_{Y,z}\cdot P_0^{(0)}\nabla_{Y,z}\phi_1^{(0)}\trace 
  + \sqrt{\mu}\nam\cdot P[\sigma]\nam \chi_1^{(1)}   \\
   \hspace{1cm}\dsp =
   - \nabla_{Y,z}\cdot P_1^{(0)} \left( \begin{matrix} \nabla\psi_0
   \\
   0\end{matrix}\right)\trace+\sqrt{\mu}\nam \cdot A + 
   \sqrt{\mu}g\trace ~,  \\
   ({\phi_1^{(0)}}_\trace+\sqrt{\mu}\chi_1^{(1)}) \,_{\vert_{z=0}}={\psi_1}_\trace ~, 
   \vspace{1mm}\\
   \dsp  -{\bf e_z}\cdot \big(P_0^{(0)}\nabla_{Y,z}\phi_1^{(0)}\trace\big)_{\vert_{z=-1}}
   -\sqrt{\mu}{\bf e_z}\cdot P[\sigma]\nam
  ( {\chi_1^{(1)}\,} _\trace )\vert_{z=-1}\\
   \hspace{1cm}
    \dsp={\bf e_z}\cdot P_1^{(0)}\left( \begin{matrix} \nabla\psi_0  \\
    0\end{matrix}\right)-\sqrt{\mu}{\bf e_z}\cdot {A}\}_{\vert_{z=-1}} ~,
\end{array}\right.
\end{equation}
with 
\[
     A = \widetilde{A} - \frac{1}{\gamma\sqrt{\mu}}P_1^{(0)}\nam
     (\phi_0-\psi_0) - \frac{1}{\sqrt{\mu}}P_1^{(1)}\nam\phi_0 
   + (2\nabla\sigma_0\cdot \nabla_Y\phi_1^{(0)} 
   - \sqrt{\mu}\frac{\abs{\nabla\sigma_0}^2}{h_0}\dz\phi_1^{(0)}){\bf e}_z
\] 
and with
\begin{eqnarray}
\label{defgbis}
  g_\trace & = & \widetilde{g}  
  - \nabla \dz \sigma_0\cdot \nabla_Y\phi_1^{(0)} 
  - (\nabla_X\zeta_1)\cdot \nabla\psi_0-(\zeta_1-b)\Delta\psi_0 ~.
\end{eqnarray}

In order to solve \eqref{eq2.16ter}, we construct $\phi_1^{(0)}$ as a 
solution of a \emph{cell problem} in the variables $Y$ and $z$ (the 
variable $X$ being considered a parameter). The resulting solution
cancels the higher order terms in \eqref{eq2.16ter}, and we are left 
with an equation for the corrector $\chi_1^{(1)}$.

\subsubsection{The cell problem}

We assume that $b, \zeta_1$ are  periodic with respect to the 
variable $Y$ and we seek a periodic function $\phi_1^{(0)}(\cdot,Y,z)$ 
that solves 
\begin{align} \label{eqv1-0}
   \left\lbrace
   \begin{array}{l}
    \nabla_{Y,z} \cdot P_0^{(0)} \nabla_{Y,z} \phi_1^{(0)} =
   - \nabla_{Y,z}\cdot P_1^{(0)} \left( \begin{matrix} \nabla\psi_0  \\
     0\end{matrix}\right)   \\
   \, \phi_1^{(0)}\,_{\vert_{z=0}}=\psi_1; \qquad 
   - {\bf e_z} \cdot P_0^{(0)}\nabla_{Y,z}{\phi_1^{(0)}}\,_{\vert_{z=-1}}
   = {\bf e_z} \cdot P_1^{(0)} \left( \begin{matrix} \nabla\psi_0
   \\
     0\end{matrix}\right) ~.
   \end{array} \right.
\end{align}
This choice of $\phi_1^{(0)}$ cancels the highest order terms in \eqref{eq2.16ter}.
Taking into account the definition of $P_0^{(0)}$ and $P_1^{(0)}$, 
we can further simplify (\ref{eqv1-0}) into
\begin{align} \label{eqv1}
    \left\lbrace
   \begin{array}{l}
     \dsp (h_0^2 \Delta_Y + \partial_z^2 )\phi_1^{(0)} =0\\
     \dsp {\phi_1^{(0)}}\,_{\vert_{z=0}}=\psi_1, \qquad \frac{1}{h_0} \partial_z
      \phi_1^{(0)}   \,_{\vert_{z=-1}}= \nabla_Y b \cdot \nabla \psi_0 .
   \end{array} \right.
\end{align}
We recall that the spaces $H^{r_1,r_2}$ that appear in the statement below are defined in (\ref{defHrr}).
\begin{prop}\label{lem2.4}
The solution $\phi_1^{(0)}$ of the cell problem \eqref{eqv1} is given in operator
notation by the expression 
\begin{equation}\label{Eqn:ExpressionForPhi_1}
   \phi_1^{(0)}(X, Y,z) = \frac{ \cosh(h_0(z+1)|D_Y|)}{\cosh(h_0|D_Y|)}
   \psi_1(X, Y) + \frac{\sinh(h_0 z|D_Y|)}{\cosh(h_0|D_Y|)} \frac{\nabla_Y}{|D_Y| }
   b(Y)\cdot \nabla \psi_0(X) ~.
\end{equation}
Assume that $h_0 = h_0(X) = 1+\eps\zeta_0$ satisfies the hypotheses
\eqref{eq12bis}, and let $r_0>d/2$ and $r\in {\mathbb N}$. Then, for all multiindex
$\alpha=(\alpha^1,\alpha^2)\in {\mathbb N}^{d}\times {\mathbb N}^d$
such that $\abs{\alpha^1}+\abs{\alpha^2}=r$, one has
\begin{equation}\label{Eqn:L2EstimateOfPhi0}
\begin{array}{lcl}
\dsp \Abs{\partial^{\alpha^1}_X\partial^{\alpha^2}_Y\nabla_X \phi_1^{(0)}}_{L^2_XL^\infty_{Y,z}}
&\leq& C(\abs{h_0}_{C^{r+1}})\big(\abs{\nabla_{X}\psi_1}_{H^{\abs{\alpha^1},r_0+\abs{\alpha^2}}}
+\abs{\nabla_X\psi_1}_{H^{0,r_0+r}}
+\abs{b}_{C^{r+2}}\abs{\nabla\psi_0}_{H^{r+1}}\big),\vspace{1mm}\\
\dsp \Abs{\partial^{\alpha^1}_X\partial^{\alpha^2}_Y\nabla_{Y,z} \phi_1^{(0)}}_{L^2_XL^\infty_{Y,z}}
&\leq& C(\abs{h_0}_{C^{r+1}})\big(\abs{\nabla_{Y}\psi_1}_{H^{\abs{\alpha^1},r_0+\abs{\alpha^2}}}
+\abs{\nabla_Y\psi_1}_{H^{0,r_0+r}}+\abs{b}_{C^{r+2}}\abs{\nabla\psi_0}_{H^{r+1}}\big).
\end{array}
\end{equation}
Moreover, derivatives of the multiscale function $\phi_1^{(0)}\,_\trace$ are
 controlled  as follows
\begin{equation}
\label{eqconseq}
\Abs{\Lambda^r \nam \phi_1^{(0)}}_{L^2}
\leq \mu^{-r/2}C(\abs{h_0}_{C^{r+1}})
 \times \big(\mu^{\frac{r+1}{2}}\abs{\nabla_X\psi_1}_{H^{r,r_0}}
+\abs{\nabla_{Y}\psi_1}_{H^{0,r_0+r}}
+\abs{b}_{C^{r+2}}\abs{\nabla\psi_0}_{H^{r+1}}\big),
\end{equation}
and if $r\geq 1$, the same upper bound holds for $\frac{1}{\sqrt{\mu}}\Abs{\Lambda^{r-1}\dz \nam\phi_1^{(0)}}_{L^2}$.
\end{prop}

It is of note that the solution of the cell problem is a multiscale
expression that can be differentiated arbitrarily many times with 
respect to the variables $(X,Y,z)$  without developing singular 
behavior in the limit as $\mu \to 0$. 

\begin{proof} 
Decomposing the function $\phi_1^{(0)}$ in Fourier modes 
$\widehat \phi_{1k}^{(0)}$ with respect to the $Y$ variable, we find
\begin{equation}
    \widehat \phi_{1k}^{(0)} = A_k e^{h_0 z|k|} +B_k e^{-h_0 z|k|}
   \nonumber
\end{equation}
with coefficients
\begin{align}
   A_k & = \frac{1}{e^{h_0|k|} + e^{-h_0|k|}} \big( i
   \frac{k}{|k|}\cdot \nabla \psi_0 \widehat b_k  + e^{h_0|k|} \widehat\psi_{1k}
    \big)    \nonumber \\
   B_k & = \frac{1}{e^{h_0|k|} + e^{-h_0|k|}} \big(- i
    \frac{k}{|k|}\cdot \nabla \psi_0 \widehat b_k  + e^{-h_0|k|} \widehat\psi_{1k}
    \big) ~.  \nonumber  
\end{align}
After substitution, the solution $\phi_1^{(0)}$ is written using operator notation
as in the statement \eqref{Eqn:ExpressionForPhi_1} of the proposition.\\
For the proof of the estimates of the derivatives of $\phi_1^{(0)}(X,Y,z)$, the expression \eqref{Eqn:ExpressionForPhi_1} is conveniently 
written in operator notation as 
\[
   C_1(h_0,z, D_Y) \psi_1(X,Y) + C_2(h_0,z,D_Y) b(Y) \cdot\nabla\psi_0(X) ~,
\]
where the components are 
\[
    C_1(h_0,z,D_Y) = \frac{ \cosh(h_0(z+1)|D_Y|)}{\cosh(h_0|D_Y|)} ~,
    \qquad 
    C_2(h_0,z,D_Y) =  \frac{\sinh(h_0z|D_Y|)}{\cosh(h_0|D_Y|)}
        \frac{\nabla_Y}{|\nabla_Y|} ~.
\]
 For the first term, we remark that the Sobolev embedding $H^{r_0}_Y\subset L^\infty_Y$ yields
$$
\Abs{\partial^{\alpha^1}_X\partial^{\alpha^2}_Y\underline{\nabla}_k C_1(h_0,z, D_Y) \psi_1}_{L^2_XL^\infty_{Y,z}}^2\lesssim \int_{\R^d}\big[\sup_{z}(\abs{(\partial^{\alpha^1}_X\partial^{\alpha^2}_Y\underline{\nabla}_k C_1\psi_1)(X,\cdot,z)}_{H^{r_0}_Y})\big]^2dX,\qquad (k=0,1)
$$
where $\underline{\nabla}_0$ stands for $\nabla_{X}$ and $\underline{\nabla}_1$ for $\nabla_{Y,z}$.\\
Now, by Plancherel formula (with respect to $Y$), one easily checks that
\begin{eqnarray*}
\sup_z \abs{(\partial^{\alpha^1}_X\partial^{\alpha^2}_Y{\nabla}_{X} C_1\psi_1)(X,\cdot,z)}_{H^{r_0}_Y}
&\leq& C(\abs{h_0}_{H^{r+1}})\Big( \sum_{\beta\leq \alpha^1 }\abs{\partial_X^\beta \nabla_X\psi_1(X,\cdot)}_{H^{r_0+r-\abs{\beta}}_Y}  + \abs{\nabla_Y \psi_1(X,\cdot)}_{H^{r_0+r}_Y}\Big),\\
\sup_z \abs{(\partial^{\alpha^1}_X\partial^{\alpha^2}_Y{\nabla}_{Y,z} C_1\psi_1)(X,\cdot,z)}_{H^{r_0}_Y}&\leq& C(\abs{h_0}_{H^{r+1}}) \sum_{\beta\leq \alpha^1 }\abs{\partial_X^\beta \nabla_Y\psi_1(X,\cdot)}_{H^{r_0+r-\abs{\beta}}_Y}
\end{eqnarray*}
Plugging these inequalities into the integral above then yields the desired result,
\begin{eqnarray*}
\Abs{\partial^{\alpha^1}_X\partial^{\alpha^2}_Y\nabla_{X} C_1(h_0,z, D_Y) \psi_1}_{L^2_XL^\infty_{Y,z}}
& \leq& C(\abs{h_0}_{C^{r+1}})\big( \abs{\nabla_{X,Y}\psi_1}_{H^{\abs{\alpha^1},r_0+\abs{\alpha^2}}}
+\abs{\nabla_{X,Y}\psi_1}_{H^{0,r_0+r}}\big),\\
\Abs{\partial^{\alpha^1}_X\partial^{\alpha^2}_Y\nabla_{Y,z} C_1(h_0,z, D_Y) \psi_1}_{L^2_XL^\infty_{Y,z}}
& \leq& C(\abs{h_0}_{C^{r+1}})  \big( \abs{\nabla_{Y}\psi_1}_{H^{\abs{\alpha^1},r_0+\abs{\alpha^2}}}
+\abs{\nabla_Y\psi_1}_{H^{0,r_0+r}}\big)  .
\end{eqnarray*}
For the control of the derivatives of $C_2(h_0,z,D_Y)b(Y)\cdot \nabla\psi_0$, we easily get that
$$
\Abs{\partial^{\alpha^1}_X\partial^{\alpha^2}_Y\nabla_{X,Y,z} C_2(h_0,z, D_Y) b(Y)\cdot \nabla\psi_0}_{L^2_XL^\infty_{Y,z}}^2\leq  C(\abs{h_0}_{C^{r+1}})
\abs{b}_{C^{r+2}}\abs{\nabla\psi_0}_{H^{r+1}},
$$
where  we have (somewhat non-optimally) estimated the action of 
singular integral operators on $L^\infty_Y$ at the cost of one 
derivative. This ends the proof of (\ref{Eqn:L2EstimateOfPhi0}).\\
For the proof of (\ref{eqconseq}), we remark that the  l.h.s. can be controlled as
$$
\Abs{\Lambda^r \nam (\phi_1^{(0)}\,_\trace)}_{L^2}\leq \sum_{\alpha^1,\alpha^2}
\mu^{-\frac{\abs{\alpha^2}}{2}}
\Abs{\partial^{\alpha^1}_X\partial^{\alpha^2}_Y (\nabla_{Y,z}+\nam_X)\phi_1^{(0)}}_{L^2_XL^\infty_{Y,z}}, 
$$
where the summation is over $(\alpha^1,\alpha^2)\in {\mathbb N}^d\times {\mathbb N}^d$ such that $\abs{\alpha^1}+\abs{\alpha^2}\leq r$. The result follows therefore from
 (\ref{Eqn:L2EstimateOfPhi0}) and a straightforward interpolation between the lowest and highest terms in terms of $\mu$. In order to  prove that the same bound holds for $\frac{1}{\sqrt{\mu}}\Abs{\Lambda^{r-1}\dz \nam\phi_1^{(0)}}_{L^2}$
when $r\geq 1$,  we just have to remark that
$$
\Abs{\Lambda^{r-1} \dz \nam (\phi_1^{(0)}\,_\trace)}_{L^2}\leq \sum_{\alpha^1,\alpha^2}
\mu^{-\frac{\abs{\alpha^2}}{2}}
\Abs{\partial^{\alpha^1}_X\partial^{\alpha^2}_Y (\nabla_{Y,z}+\nam_X)\dz\phi_1^{(0)}}_{L^2_XL^\infty_{Y,z}},
$$
where the summation is over $(\alpha^1,\alpha^2)\in {\mathbb N}^d\times {\mathbb N}^d$ such that $\abs{\alpha^1}+\abs{\alpha^2}\leq r-1$. From the explicit expression of $\phi_1^{(0)}$, we can check it is possible to replace
$\dz \phi_1^{(0)}$ by $\abs{D_Y}\phi_1^{(0)}$ in the above summation, so that the result follows as for
 (\ref{eqconseq}).
\end{proof}

\subsubsection{Estimate on the corrector  $\chi_1^{(1)}$}

With $\phi_1^{(0)}$ as in the previous section, the system
(\ref{eq2.16ter}) reduces to the following boundary value 
problem for $\chi_1^{(1)}$, 
\begin{equation}\label{eq3.17new}
   \left\lbrace
   \begin{array}{l}
     \dsp \nam\cdot P[\sigma]\nam \chi_1^{(1)} =\nam \cdot
     A + g  ~,   \\
     \chi_1^{(1)} \,_{\vert_{z=0}}=0,\qquad
     \dsp 	
    - {\bf e_z}\cdot P[\sigma]\nam \chi_1^{(1)}\,_{\vert_{z=-1}}
   = - {\bf e_z}\cdot {A} \,_{\vert_{z=-1}}  ~.
   \end{array}\right.
\end{equation}

\begin{prop}\label{lemmchi}
Let $r\in {\mathbb N}$ and denote $(r-1)_+=\max\{r-1,0\}$ and
$\widetilde{r}=(r-1)_++1$.  The solution $\chi_1^{(1)}$ of
\eqref{eq3.17new} satisfies the estimates 
\begin{eqnarray*}
 \Vert\Lambda^r\nam\chi_1^{(1)}\Vert_{L^2} & \leq & \mu^{-\frac{r}{2}}M_r
   \big(\abs{\nabla\psi_0}_{H^{r+1}}+\mu^{r/2} 
   {\abs{\nabla_{X,Y}\psi_1}_{H^{r+1,r_0}}+\abs{\nabla_{X,Y}\psi_1}_{H^{1,r_0+r}}\big)},\\ 
 \Vert\Lambda^{r-1}\dz \nam\chi_1^{(1)}\Vert_{L^2} & \leq & \mu^{-\frac{(r-1)_+}{2}}M_r
   \big(\abs{\nabla\psi_0}_{H^{r+1}}+\mu^{\frac{(r-1)_+}{2}} 
     {\abs{\nabla_{X,Y}\psi_1}_{H^{\widetilde{r}+1,r_0}} 
   + \abs{\nabla_{X,Y}\psi_1}_{H^{1,r_0+\widetilde{r}}}\big)} 
   ~, 
\end{eqnarray*}
with
$M_r = C(\frac{1}{\alpha},\frac{1}{\alpha_0},
  \abs{\zeta_0}_{C^{r+1}\cap C^2},\abs{\zeta_1}_{C^{r+1}\cap C^2},
 \abs{b}_{C^{r+1}\cap C^2})$.
\end{prop}
\begin{proof}
The method follows the recipe of classical energy estimates, 
paying attention to the rapidly oscillating coefficients and 
their commutators with differential operators. Indeed
multiplying (\ref{eq3.17new}) by $ \chi_1^{(1)}$ and integrating by parts yields
\begin{eqnarray*}
  \int_{\Omega_0}P[\sigma]\nam \chi_1^{(1)}\cdot\nam \chi_1^{(1)} \,dzdX  = \int_{\Omega_0}A\cdot \nam\chi_1^{(1)}\, dzdX 
     - \int_{\Omega_0}g \chi_1^{(1)}\, dzdX  ~.
\end{eqnarray*}
The matrix of coefficients $P[\sigma]$ is coercive under condition 
\eqref{eq12} (Proposition 2.3(iii) of \cite{AL}), and is uniformly 
so with regard to the small parameters, as the scaling regime we 
are studying imposes that $\beta = \gamma$.\\
Using the Cauchy-Schwarz inequality and Poincar\'e inequality, 
one finds, as in the proof of Proposition \ref{lemm3}, that
\begin{equation}\label{annivaiantze}
  \Vert\nam\chi_1^{(1)}\Vert_{L^2} \leq M_0
   \big(\Vert A\Vert_{L^2} + \Vert g\Vert_{L^2} \big)
\end{equation}
with $M_0$ as in the statement of the proposition.\\
For the general case $r\in {\mathbb N}$, the procedure is exactly the same replacing
$g$ by $\Lambda^r g$  and
$A$ by $A^{(r)}$, with
\[
A^{(r)}=\Lambda^r A-[\Lambda^r,P[\sigma]]\nam\chi_1^{(1)};
\]
in particular, it follows from classical commutator estimates and Lemma \ref{lemmeP1} that
\[
\Vert A^{(r)}\Vert_{L^2}\leq \Vert \Lambda^r
A\Vert_{L^2}+C(\frac{1}{\alpha_0},\frac{1}{\alpha},\vert
\zeta_0\vert_{C^{r+1}},\vert \zeta_1\vert_{C^{r+1}},\vert
b\vert_{C^{r+1}})\sum_{k=1}^{r} \mu^{-k/2}\Vert
\Lambda^{r-k}\nam\chi_1^{(1)}\Vert_{L^2} 
\]
so that (\ref{annivaiantze}) yields in this configuration
\begin{equation}\label{estchi}
   \Vert\Lambda^r\nam\chi_1^{(1)}\Vert_{L^2} \leq M_0
   \big(\Vert \Lambda^r A\Vert_{L^2} + \Vert \Lambda^r g\Vert_{L^2} \big)
   + M_r\sum_{k=1}^r \mu^{-k/2}\Vert \Lambda^{r-k}\nam\chi_1^{(1)}\Vert_{L^2} ~.
\end{equation}
We therefore need the following lemma.
\begin{lemm}\label{lemmcontrolA}
The following estimate holds,
\[
    \Abs{\Lambda^r A}_{L^2} +\Abs{\Lambda^r g}_2
  \leq \mu^{-r/2}M_r\big(\abs{\nabla\psi_0}_{H^{r+1}}+
  \mu^{\frac{r}{2}}(\abs{\nabla_{X}\psi_1}_{H^{r,r_0}}+\abs{\nabla_{Y}\psi_1}_{H^{r+1,r_0}})
+\abs{\nabla_{X}\psi_1}_{H^{0,r_0+r}} 
+
\abs{\nabla_{Y}\psi_1}_{H^{1,r_0+r}}\big)
\]
\end{lemm}

\begin{proof}
- \emph{Control of $\Vert \Lambda^r A\Vert_{L^2}$.} Recall that
\begin{eqnarray}
\nonumber 
   A &=& - P_1\nam\phi_1^{(0)} - \frac{1}{{\mu}}
      P_1^{(0)}\nam (\phi_0-\psi_0) - \frac{1}{\sqrt{\mu}}P_1^{(1)}\nam\phi_0\\
& &- P_0\left(\begin{array}{c}\nabla_X\phi_1^{(0)}\\0\end{array}\right)_\trace
\label{Eqn:DefinitionOfA}
+\big(2\nabla\sigma_0\cdot \nabla_Y
\phi_1^{(0)}-\sqrt{\mu}\frac{\abs{\nabla\sigma_0}^2}{h_0}\dz
\phi_1^{(0)}\big)_\trace{\bf e}_z ~. 
\end{eqnarray}
A direct application of the chain rule gives
\begin{eqnarray}
\nonumber
   \Abs{\Lambda^r A}_{L^2} &\leq&
   \sum_{k=0}^r\Big[\Abs{P_1}_{L^\infty_zW^{k,\infty}}
   \Abs{\Lambda^{r-k} \nam\phi_1^{(0)}}_{L^2}  
   + \Abs{P_1^{(0)}}_{L^\infty_zW^{k,\infty}}
   \Abs{\frac{1}{\mu}\Lambda^{r-k}\nam (\phi_0-\psi_0)}_{L^2} \\ 
\nonumber
   & &\qquad + \Abs{P_1^{(1)}}_{L^\infty_zW^{k,\infty}}
   \Abs{\frac{1}{\sqrt{\mu}}\Lambda^{r-k} \nam \phi_0}_{L^2}\Big]\\ 
& &
\label{Eqn:A-Estimate}
+\Abs{P_0}_{L^\infty_zW^{r,\infty}}\Abs{\Lambda^r\nabla_X\phi_1^{(0)}}_{L^2}
+C(\abs{\zeta_0}_{C^{r+1}},\frac{1}{\alpha_0})\Abs{\Lambda^r\nabla_{Y,z}\phi_1^{(0)}}_{L^2}.
\end{eqnarray}
Using (\ref{eq15}) and Lemma \ref{lemmeP1} to control the norms of
$P_0$ and $P_1$, and Proposition \ref{lemm3} to control 
the second and third terms in the above expression, we find
\[
   \Abs{\Lambda^r A}_{L^2} \leq
   M_r\big(\mu^{-r/2}\abs{\nabla\psi_0}_{H^{r+1}} 
   + \sum_{k=0}^r \mu^{-k/2}\Abs{\Lambda^{r-k} \nam \phi_1^{(0)}}_2  
   + \Abs{\Lambda^r(\nabla_X{\phi_1^{(0)}}_\trace)}_2 
   + \Abs{\Lambda^r(\nabla_{Y,z}{\phi_1^{(0)}}_\trace)}_2\big). 
\]
We now control $\Abs{\Lambda^{r-k} \nam \phi_1^{(0)}}_2 $ through (\ref{eqconseq}), while 
$\Abs{\Lambda^r(\nabla_X\phi_1^{(0)}\,_\trace)}_2$ and
$\Abs{\Lambda^r(\nabla_{Y,z}\phi_1^{(0)}\,_\trace)}_2$ 
can be controled using (\ref{Eqn:L2EstimateOfPhi0}) and proceeding as
in the proof of (\ref{eqconseq}). This yields 
\[
    \Abs{\Lambda^r A}_{L^2}
  \leq \mu^{-r/2}M_r\big(\abs{\nabla\psi_0}_{H^{r+1}}+
  \mu^{\frac{r}{2}}\abs{\nabla_{X}\psi_1}_{H^{r,r_0}}+\abs{\nabla_{X,Y}\psi_1}_{H^{0,r_0+r}}\big) 
  ~. 
\]
- \emph{Control of $\Abs{\Lambda^r g}_{L^2}$.} We first recall that
\[
  g = -\left(\begin{array}{c}\nabla_X\\0\end{array}\right) 
\cdot P_0 (\nabla_{Y,z} \phi_1^{(0)})_\trace  +
\Big(-\nabla\dz\sigma_0\cdot \nabla_Y
\phi_1^{(0)}-\nabla_X\zeta_1\cdot
\nabla\psi_0-(\zeta_1-b)\Delta\psi_0\Big)_\trace ~. 
\]
We get therefore
\begin{eqnarray*}
\nonumber
\Abs{\Lambda^r g}_2&\leq&
\mu^{-r/2}M_r\big(\abs{\nabla\psi_0}_{H^{r+1}}+\mu^{r/2}\Abs{\Lambda^r
  (\nabla_{Y,z}\phi_1^{(0)}\,_\trace)}_{H^1}\big)\\ 
&\leq &
   \mu^{-r/2}M_r\big(\abs{\nabla\psi_0}_{H^{r+1}}+\mu^{r/2}
   \abs{\nabla_{Y}\psi_1}_{H^{r+1,r_0}} + \abs{\nabla_{Y}\psi_1}_{H^{1,r_0+r}}
   \big). 
\end{eqnarray*}
\end{proof}
Using (\ref{estchi}) and the lemma, we get directly
\begin{eqnarray*}
 \mu^{r/2}\Vert\Lambda^r\nam\chi_1^{(1)}\Vert_{L^2}& \leq&  M_r
   \big(\abs{\nabla\psi_0}_{H^{r+1}}+\mu^{r/2} 
   \abs{\nabla_{X,Y}\psi_1}_{H^{r+1,r_0}} +
   \abs{\nabla_{X,Y}\psi_1}_{H^{1,r_0+r}}\big) \\
  & &
   + M_r  \sum_{k=1}^{r}\mu^{(r-k)/2}\Vert \Lambda^{r-k}\nam\chi_1^{(1)}\Vert_{L^2},
\end{eqnarray*}
and the estimate on $\Abs{\Lambda^r \nam \chi_1^{(1)}}$ follows from a
straightforward induction. We now turn 
to prove the estimate on  $\Abs{\Lambda^{r-1} \dz \nam
  \chi_1^{(1)}}$. As for the control of $\Lambda^{r-1}\dz\nam \phi_0$ 
in Proposition \ref{lemm3}, it is enough to get an upper bound on
$\Abs{\Lambda^{r-1} \dz^2 \chi_1^{(1)}}$. The idea is 
to proceed as in Proposition \ref{lemm3}, using the equation to get $\dz^2
\chi_1^{(1)}$ in terms of quantities under control, 
\[ \partial^2_z\chi_1^{(1)}
     = \frac{h}{1+\mu\vert \nabla\sigma\vert^2}\Big[\mu
     \bigl(-\frac{\dz\vert\nabla\sigma
       \vert^2}{h}\dz\chi_1^{(1)}+\partial_z (\nabla \sigma\cdot
     \nabla \chi_1^{(1)})  
        + \nabla\cdot(\nabla\sigma \partial_z\chi_1^{(1)}) 
        - \nabla\cdot( h\nabla \chi_1^{(1)}) \bigr)    
        + RHS\eqref{eq3.17new}\Big] ~;
\]
the presence of the fast scale $X/\sqrt{\mu}$ makes things a little
more complicated than in the proof of  
Proposition \ref{lemm3}, and we need product estimates and a refinement of
Lemma \ref{lemmfrac} for multiscale functions. 
\begin{lemm}\label{lemfrac2}
Let $r\in {\mathbb N}$, and denote $(r-1)_+=\max\{r-1,0\}$. Let also
$G=G(X,Y,z)\in L^\infty_zW^{(r-1)_+,\infty}_{X,Y}$, and  $F\in
L^2(\Omega_0)$ be such that $\Lambda^{(r-1)_+} F\in
L^2(\Omega_0)$. Then 
\begin{eqnarray*}
 \Abs{\Lambda^{r-1} (G_\trace F)}_{L^2}&\lesssim& \Vert G\Vert_{L^\infty_z W^{(r-1)_+,\infty}_{X,Y}}
\sum_{k=0}^{{(r-1)_+}}\mu^{-k/2}\Vert \Lambda^{(r-1)_+-k} F\Vert_{L^2}.
\end{eqnarray*}
If moreover $G\geq 0$, then
\begin{eqnarray*}
\Abs{\Lambda^{r-1}\frac{F}{1+G_\trace}}_{L^2}&\leq&  C(\Vert
G\Vert_{L^\infty_z
  W^{(r-1)_+,\infty}_{X,Y}})\sum_{k=0}^{(r-1)_+}\mu^{-k/2}\Vert
\Lambda^{(r-1)_+-k} F\Vert_{L^2}. 
\end{eqnarray*}
\end{lemm}
\begin{proof}
The product estimates are a straightforward consequence of the chain rule;
the second pair of estimates is derived as in Lemma \ref{lemmfrac}
using these product estimates. 
\end{proof}
Using Lemma \ref{lemfrac2} and the above expression for
$\dz^2\chi_1^{(1)}$, we get, with $\widetilde{r}=(r-1)_++1$, 
\begin{eqnarray}
\label{bourg1}
\Abs{\Lambda^{r-1}\dz^2\chi_1^{(1)}}_{L^2}&\leq& M_r
\sum_{k=0}^{(r-1)_+}\mu^{-k/2}\big(\sqrt{\mu}\Abs{\Lambda^{\widetilde{r}-k}
  \nam
  \chi_1^{(1)}}_2+\Abs{\Lambda^{(r-1)_+-k}RHS\eqref{eq3.17new}}_2\big). 
\end{eqnarray}
We can now use Lemma \ref{lemmcontrolA} to get
\begin{eqnarray}
\nonumber
\lefteqn{\Abs{\Lambda^{(r-1)_+-k}RHS\eqref{eq3.17new}}_{L^2} 
   \leq
   \sqrt{\mu}\Abs{\Lambda^{\widetilde{r}-k}A}_{L^2}+\Abs{\Lambda^{(r-1)_+-k}g}_{L^2}
   + \Abs{\Lambda^{(r-1)_+-k}\dz 
    A}_{L^2}}\\ 
\label{bourg2}
&\leq&
\mu^{-\frac{\widetilde{r}-k-1}{2}}M_r\big(\abs{\nabla\psi_0}_{H^{r+1}}
  +
  \mu^\frac{\widetilde{r}}{2}\abs{\nabla_{X,Y}\psi_1}_{H^{\widetilde{r},r_0}}
  + \abs{\nabla_{X,Y}\psi_1}_{H^{0,\widetilde{r}+r_0}}\big)+\Abs{\Lambda^{(r-1)_+-k}\dz 
  A}_{L^2}, 
\end{eqnarray}
so that the only thing we still need to prove is a control 
on  $\Abs{\Lambda^{(r-1)_+-k}\dz A}_{L^2}$.
\begin{lemm}
For all $m\in {\mathbb N}$,  $0\leq m\leq r-1$, one has
$$
\Abs{\Lambda^{m}\dz A}_{L^2}\leq
\mu^{-\frac{m}{2}}M_r\big(\abs{\nabla\psi_0}_{H^{r+1}}+\mu^{m/2}
\abs{\nabla_{X,Y}\psi_1}_{H^{m+1,r_0}}+\abs{\nabla_{X,Y}\psi_1}_{H^{0,r_0+m+1}}
\big). 
$$ 
\end{lemm}
\begin{proof}
\noindent From the explicit expression of $A$ provided by
(\ref{Eqn:DefinitionOfA}), and proceeding as in the proof of Lemma
\ref{lemmcontrolA}, we get 
$$
\Abs{\Lambda^{m}\dz A}_{L^2}\leq M_r\big(\mu^{-\frac{m}{2}}\abs{\nabla\psi_0}_{H^{r+1}}+\Abs{\Lambda^{m}(\nabla_{X,Y,z}\phi_1^{(0)}\,_\trace)}_{H^1_zL^2}
+\sum_{k=0}^{m}\mu^{-k/2}\Abs{\Lambda^{m-k}\dz \nam\phi_1^{(0)}}_{H^1_zL^2}\big),
$$
and the result follows from Proposition \ref{lem2.4}.
\end{proof}
The desired control on $\Abs{\Lambda^{r-1}\dz^2\chi_1^{(1)}}_{L^2}$ is
then a direct consequence of (\ref{bourg1}), (\ref{bourg2}) 
and the lemma.
\end{proof}

\subsection{Asymptotic expansion of the Dirichlet-Neumann 
operator with estimates}

We study here the asymptotic behavior of the Dirichlet-Neumann operator,
\begin{equation}\label{rappelG}
   G[ \zeta, \beta b_\gamma] \psi 
   =  {\bf e_z} \cdot P[\sigma] \nam \phi\, _{\vert_{z=0}} ~,
\end{equation}
where $\phi$ solves \eqref{eq14}. In the previous section, we have 
shown that when the surface parametrization $\zeta$ and the trace 
of the potential at the surface $\psi$ are of the form \eqref{ansatz}, 
one can decompose $\phi$ into
\[
   \phi = \phi_0 + \mu\chi,
\]
where $\phi_0$ is deduced from \eqref{eq14} by neglecting all the 
contributions due to the roughness. We have further decomposed 
$\phi_0$ and the residual $\chi$ (which contains all the roughness 
effects) as
\[
   \phi_0 = \psi_0+\mu\phi_0^{(1)}+\mu^2\chi_0^{(2)}
       \quad\mbox{ and }\quad
   \chi = \phi_1^{(0)}+\sqrt{\mu}\chi_1^{(1)},
\]
with controls on the residuals $\chi_0^{(2)}$ and $\chi_1^{(1)}$ given 
in Propositions \ref{lemm3} and \ref{lemmchi} respectively. We can therefore 
rewrite $\phi$ as the sum of an {\it effective} part and a
 {\it residual} part,
\[
   \phi = \phi_{\hbox{\small \itshape eff}} + \mu^{3/2}\phi_{res}
\]
with
\[
   \phi_{\hbox{\small \itshape eff}} = \psi_0+\mu(\phi_0^{(1)}+\phi_1^{(0)})
     \quad\mbox{ and }\quad
   \phi_{res} = \chi_1^{(1)}+\sqrt{\mu}\chi_0^{(2)} ~.
\]
Similarly, we decompose the Dirichlet-Neumann operator into an
effective and residual part as follows:

\begin{prop}\label{Lemma3.7}
We separate the  expression for the Dirichlet-Neumann operator as   
\[
    G[ \zeta,\beta b_\gamma] \psi = (G\psi)_{\hbox{\small \itshape eff}} + 
    \mu^{3/2} (G\psi)_{res}
\]
where
\begin{eqnarray*} 
   \frac{1}{\mu} (G\psi)_{\hbox{\small \itshape eff}} 
   & = & -\nabla\cdot (h_0\nabla\psi_0)  
     - \nabla_Y \zeta_1\,_{\trace} \cdot \nabla \psi_0\cr
   & & + |D_Y|\tanh(h_0 |D_Y|) \psi_1(X,Y) _\trace 
       - \nabla\psi_0 \cdot \nabla_Y (\sech(h_0|D_Y) b)_\trace ~.
\end{eqnarray*}
{The remainder $(G\psi)_{res}$ satisfies the estimate  ($r$  integer), 
\begin{equation}\label{Eqn:G-RemainderEstimate}
   |{(G\psi)_{res}}|_{H^{r}}  \leq M_r \mu^{-r/2 -1/8} \Big(\abs{\nabla\psi_0}_{H^{r+3}}
  + \mu^{r/2} \abs{\nabla_{X,Y}\psi_1}_{H^{r+1,r_0}} + \abs{\nabla_{X,Y}\psi_1}_{H^{1,r_0+r}}
\Big)
\end{equation}}

\end{prop}

\begin{rema}
The nonlocal operators  in the expression for $(G\psi)_{eff}$ arise from the simultaneous   homogenization process and shallow water limit.
Homogenization analysis on a shallow water expansion would give a different result. The reason for
 this difference is that certain terms neglected in standard shallow water expansions are not negligible in
 the presence of rapidly varying bathymery;  their effects are described in  the nonlocal terms of 
$(G\psi)_{eff}$.
\end{rema}

\begin{proof}
Recalling that $P[\sigma]=P_0^{(0)}+\sqrt{\mu}(P_0^{(1)}+P_1)$, 
with $P_1=P_1^{(0)}+\sqrt{\mu}P_1^{(1)}$, one has

\begin{eqnarray*}
   G[\zeta,\beta b_\gamma]\psi 
   & = & {\bf e_z}\cdot P_0^{(0)}\nam\phi_{\hbox{\small \itshape eff}}\,_{\vert_{z=0}}
    + \sqrt{\mu}{\bf e_z}\cdot
        (P_0^{(1)}+P_1)\nam \phi_{\hbox{\small \itshape eff}}\,_{\vert_{z=0}}
    \\ \nonumber
       & & + \mu^{3/2} {\bf e_z}\cdot P[\sigma]\nam\phi_{res}\,_{\vert_{z=0}} 
\\ \nonumber
&=& {\bf e_z}\cdot P_0^{(0)}\nam\phi_{\hbox{\small \itshape eff}}\,_{\vert_{z=0}}
    + \sqrt{\mu}{\bf e_z}\cdot
  (P_0^{(1)}+P_1^{(0)})\nam\psi_0 \,_{\vert_{z=0}} +\mu 
  {\bf e_z}\cdot P_1^{(1)} \nam \psi_0 \,_{\vert_{z=0}} \nonumber \\
  && + \mu^{3/2} {\bf e_z}\cdot 
    (P_0^{(1)}+P_1)\nam (\phi_0^{(1)}+\phi_1^{(0)})\,_{\vert_{z=0}}
   + \mu^{3/2} {\bf e_z}\cdot P[\sigma]\nam\phi_{res}\,_{\vert_{z=0}}.
\end{eqnarray*}
We now decompose
\begin{eqnarray*}
   G[\eps\zeta,\beta b_\gamma]\psi 
   & = &  (G\psi)_{\hbox{\small \itshape eff}} + \mu^{3/2}(G\psi)_{res} ~,
\end{eqnarray*}
with
\begin{eqnarray}
\nonumber
   (G\psi)_{\hbox{\small \itshape eff}} 
    & = & {\bf e_z}\cdot P_0^{(0)}\nam\phi_{\hbox{\small \itshape eff}}\,_{\vert_{z=0}}
      + \sqrt{\mu}{\bf e_z}\cdot
        (P_0^{(1)}+P_1^{(0)})\nam\psi_0\,_{\vert_{z=0}} ~,  \label{defGpsi} \\ 
   (G\psi)_{res} & = & \frac{1}{\sqrt{\mu}}{\bf e_z}
     \cdot P_1^{(1)}\nam\psi_0\,_{\vert_{z=0}}
     + {\bf e_z} \cdot (P_0^{(1)}+P_1)\nam(\phi_0^{(1)}+\phi_1^{(0)})\,_{\vert_{z=0}}
\label{defGr}   \\
   & & + {\bf e_z}\cdot P[\sigma]\nam\phi_{res}\,_{\vert_{z=0}} ~. \nonumber 
\end{eqnarray}
The two tasks of this proposition are to give an expression for 
$(G\psi)_{\hbox{\small \itshape eff}}$ and to prove an estimate for 
$(G\psi)_{r}$. 

\smallskip\noindent
- \emph{Explicit computation of $(G\psi)_{\hbox{\small \itshape eff}}$}. From the 
definition of $\phi_{\hbox{\small \itshape eff}}$, we have
\begin{eqnarray}\label{expansionG} 
   \frac{1}{\mu}(G\psi)_{\hbox{\small \itshape eff}} & = &  
   \frac{1}{\mu}{\bf e_z} \cdot (P_0+\sqrt{\mu}P_1^{(0)})\nam\psi_0 \, _{\vert_{z=0}}
   + {\bf e_z}\cdot P_0^{(0)} \nam \phi_0^{(1)}\,_{\vert_{z=0}}
   + {\bf e_z}\cdot P_0^{(0)}
    \nam \phi_1^{(0)}\,_{\vert_{z=0}}    \nonumber\\
   & = & -\nabla\cdot (h_0\nabla\psi_0) - \nabla_Y \zeta_1 \trace \cdot \nabla \psi_0
   +\frac{1}{h_0} \partial_z \phi_1^{(0)} \vert_{z=0} ~.
\end{eqnarray}
Computing the last term in the RHS with the help of Proposition \ref{lem2.4}, we get
\[
   \frac{1}{h_0} \partial_z \phi_1^{(0)} \vert_{z=0} 
   = |D_Y|\tanh(h_0 |D_Y|) \psi_1(X,Y) _\trace +
     \nabla\psi_0 \cdot \nabla_Y (\sech(h_0|D_Y|) b)_\trace ,
\]
so that $G_{\hbox{\small \itshape eff}}$ is indeed given by the expression
stated in the proposition.

\smallskip\noindent
- \emph{Control of the residual $(G\psi)_{res}$.}  From the explicit
expression (\ref{defGr}) of $(G\psi)_{res}$ and 
Proposition \ref{lemm3} and \ref{lem2.4} and  Lemma \ref{lemmeP1}, we get
\begin{eqnarray*}
\abs{(G\psi)_{res}}_{H^r}&\leq& \mu^{-r/2} M_r\big(\abs{\nabla\psi_0}_{H^{r+1}}+\mu^{r/2}
{\abs{\nabla_{X,Y}\psi_1}_{H^{r+1,r_0}}+\abs{\nabla_{X,Y}\psi_1}_{H^{1,r_0+r}}\big)}\\
& &+ M_r \sum_{k=0}^r \mu^{-k/2}\abs{ \Lambda^{r-k}\nam
  \phi_{res}\,_{\vert_{z=0}}}_{L^2} ~,
\end{eqnarray*}
which motivates the following lemma.
\begin{lemm}
\[
   \abs{\nam\phi_{res}\,_{\vert_{z=0}}}_{L^2}\leq
   \mu^{-1/8}M_1\big(\abs{\nabla\psi_0}_{H^2} 
  + \mu^{1/2}{\abs{\nabla_{X,Y}\psi_1}_{H^{2,r_0}} 
  + \abs{\nabla_{X,Y}\psi_1}_{H^{1,1+r_0}}\big)} 
\]
\end{lemm}
\begin{proof}
 We write 
\begin{eqnarray*}
\abs{\nam\phi_{res}\,_{\vert_{z=0}}}_{L^2}&\lesssim& \mu^{1/8}
\abs{\nam\phi_{res}\,_{\vert_{z=0}}}_{H^{1/2}} 
  + \mu^{-1/8}\abs{\nam\phi_{res}\,_{\vert_{z=0}}}_{H^{-1/2}} 
\\
&\lesssim& \mu^{1/8}\big(\mu^{1/4}\Abs{\Lambda
  \nam\phi_{res}}_{L^2}+\mu^{-1/4}\Abs{\dz
  \nam\phi_{res}}_{L^2}\big)+\mu^{-1/8}\big(\Abs{\nam\phi_{res}}_{L^2}
  + \Abs{\Lambda^{-1}\dz\nam\phi_{res}}_{L^2}\big)  ~,
\end{eqnarray*}
where, for the second inequality, we have used two different version
of the trace lemma, namely, $|{F_{\vert_{z=0}}}|_{L^2}\lesssim
\mu^{1/4}\Abs{\Lambda^{1/2}F}_{L^2}+\mu^{-1/4}\Abs{\Lambda^{-1/2} \dz
  F}_{L^2}$ and  
 $|{F_{\vert_{z=0}}}|_{L^2}\lesssim
 \Abs{\Lambda^{1/2}F}_{L^2}+\Abs{\Lambda^{-1/2} \dz F}_{L^2}$. The
 estimate follows therefore 
 from the definition of $\phi_{res}$ and Proposition \ref{lemmchi}.
\end{proof}
{Bound for $ \abs{ \Lambda^{r-k}\nam
    \phi_{res}\,_{\vert_{z=0}}}_{L^2} $ are obtained in the same
  manner, giving rise to a power of $\mu$ in the form
  $\mu^{-1/8-(r-k)/2}$.}  
\end{proof}

Later in the consistency analysis, we will need an estimate of
$\abs{(G\psi)_{res}}_{H^{1/2}}$. For this purpose,  
interpolating between $H^{r-1}$ and $H^{r}$, we have ($r\ge 1$):
\begin{eqnarray}
   \abs{(G\psi)_{res}}_{H^{r-\frac{1}{2}}} &\leq& \mu^{1/4}  \abs{(G\psi)_{res}}_{H^{r}}  +
\mu^{-1/4}  \abs{(G\psi)_{res}}_{H^{r-1}} \cr
&\leq&  M_r \mu^{-(r-\frac{1}{2})/2 -1/8} \Big(\abs{\nabla\psi_0}_{H^{r+3}}
+\mu^{r/2} 
\abs{\nabla_{X,Y}\psi_1}_{H^{r+1,r_0}}+\abs{\nabla_{X,Y}\psi_1}_{H^{1,r_0+r}}
\Big)
\end{eqnarray}

\section{Homogenization with estimates of the 
   equations for water waves}\label{secthom}

Up to this point the canonical variables of the water wave problem
$(\zeta,\psi)$ have been treated as data for an elliptic partial 
differential equation in a fixed domain. We now return to study the
dynamics of the water waves system \eqref{eq7} as a time dependent 
problem. We first derive the effective PDEs satisfied by
$(\zeta_0,\psi_0)$ and $(\zeta_1,\psi_1)$, after which we investigate
the precision of our approximate solution \eqref{ansatz} in satisfying
the full water wave equation.  

\subsection{Effective equations}\label{sect-effective-eqs}
Consider the decomposition \eqref{Eqn:BasicAnsatz1}\eqref{Eqn:BasicAnsatz2} 
to be an Ansatz for the full Euler equations \eqref{eq7}, for 
which the slow and fast scale variables are identified through $Y =
X/\sqrt{\mu}$ and $\tau = t/\sqrt{\mu}$. Substituting these
expressions into the full equations and using the rigorous expansion
\eqref{expansionG}, for the first component of  equations of
\eqref{eq7} we obtain an equation for the quantity $(\zeta_a,\psi_a)
= (\zeta_0 + \sqrt{\mu} \zeta_1, \psi_0 +\mu \psi_1)$; 
\begin{eqnarray}\label{eq4.1}
   && \partial_t\zeta_0 + \partial_\tau \zeta_1 + h_0 \Delta \psi_0
    + \nabla \zeta_0 \cdot \nabla \psi_0 + \nabla_Y \zeta_1 \cdot
    \nabla \psi_0 \nonumber \\
   && \qquad - |D_Y|\tanh(h_0 |D_Y|) \psi_1(X,Y)_\trace 
    -\nabla\psi_0 \cdot \nabla_Y (\sech(h_0|D_Y|))b_\trace 
    = -\sqrt{\mu}\partial_t \zeta_1 + \sqrt{\mu} (G \psi_a)_{res} ~.
\end{eqnarray}
At this point the fast and slow time scales are identified, $Y =
X/\sqrt{\mu}$ and $\tau = t/\sqrt{\mu}$, and we have made no
approximation. Similarly, for the second equation of \eqref{eq7}  
we obtain
\begin{eqnarray}\label{eq4.2}
   &&  \partial_t \psi_0 + \sqrt{\mu}\partial_\tau \psi_1 + \mu\partial_t\psi_1 
       + \zeta_0 + \sqrt{\mu}\zeta_1  
       + \frac{1}{2} |\nabla \psi_0 +\sqrt{\mu}
       (\nabla_Y+\sqrt{\mu} ~\nabla_X) \psi_1|^2 =    \nonumber \\
   && \qquad \mu\frac{ \Big(  \frac{1}{\mu} 
   \big(   (G \psi_a)_{\hbox{\small \itshape eff}} +  \mu^{3/2} (G \psi_a)_{res} \big)
      + (\nabla \zeta_0 + \nabla_Y\zeta_1+\sqrt{\mu}\nabla_X\zeta_1)\cdot
      (\nabla \psi_0 +\sqrt{\mu} \nabla_Y\psi_1 +\mu \nabla_X\psi_1)
      \Big)^2}{2 \big(1 + 
      \mu|\nabla\zeta_0 + (\nabla_Y + \sqrt{\mu} ~\nabla_X)\zeta_1|^2\big) } ~.
\end{eqnarray}
Isolating the error terms in \eqref{eq4.2} onto the RHS, one obtains
\begin{eqnarray}\label{eq4.3}
   &&  \partial_t\psi_0 + \sqrt{\mu}\partial_\tau \psi_1 
       + \zeta_0 + \sqrt{\mu}\zeta_1  
       + \frac{1}{2}  |\nabla \psi_0|^2  +\sqrt{\mu}
       \nabla\psi_0 \cdot \nabla_Y \psi_1 \nonumber \\
   && \qquad = -\mu\partial_t\psi_1 
      - \mu  \nabla\psi_0 \cdot \nabla_X \psi_1 - \frac{1}{2}\mu
        |\nabla_Y\psi_1 +\sqrt{\mu} ~\nabla_X \psi_1|^2 \\
   && \qquad + \mu\frac{ \Big( \frac{1}{\mu} 
      \big( ( G \psi_a)_{\hbox{\small \itshape eff}} + \mu^{3/2} (G \psi_a)_{res} \big)
      + (\nabla \zeta_0 + \nabla_Y\zeta_1 + \sqrt{\mu}\nabla_X\zeta_1)\cdot
      (\nabla \psi_0 +\sqrt{\mu} \nabla_Y\psi_1 + \mu \nabla_X\psi_1)
      \Big)^2}{2 \big(1 + 
      \mu|\nabla\zeta_0 + (\nabla_Y + \sqrt{\mu} ~\nabla_X)\zeta_1|^2\big)} ~.
       \nonumber 
\end{eqnarray}
The bottom profile  $b$ is function of the fast variables $Y = X/\sqrt{\mu}$. 

Now adopt the point of view that we seek multi-scale approximations to
the system of equations \eqref{eq4.1}\eqref{eq4.3}. To impose this
scaling regime, make the assumption that the variables $t,X$ and $\tau,Y$
are independent, so that $(\zeta_0 +\sqrt{\mu} \zeta_1, \psi_0 +
{\mu}\psi_1)$ are multiscale functions of the variables $(t,X,\tau,Y)$. 
In these equations, $\zeta_0$ and $\psi_0$ are functions of $X,t $
only, while $\zeta_1$ and $\psi_1$ are multi-scale functions of both
time and space. The original variables will be re-imposed when we return
to the identification $Y=X/\sqrt{\mu}$ and $\tau=t/\sqrt{\mu}$. In order
to justify this otherwise formal separation of slow and fast scales, we
will use results on scale separation that appear in \cite{BLP}\cite{CGNS}. 

Equations \eqref{eq4.1}\eqref{eq4.2} are two equations for the four 
unknown quantities $(\zeta_0, \zeta_1, \psi_0, \psi_1)$.  In order to 
obtain well defined evolution equations for them, one must identify
dynamics that take place on the slow and fast space and time scales. 
This is performed using the following scale separation lemmas, in
which the underlying periodic nature of the cell problem plays a r\^ole. 

\begin{prop}\label{sc-sep1} Let $g$ be a  continuous 
function on $\R^d$ which is
periodic over  $\T^d$,  and denote 
$\bar g=\frac{1}{(2\pi)^d}
\int_{\T^d} g(Y) dY $ its average value on $\T^d$. For any function $f(X)$  
in the Schwarz space ${\mathcal S}(\R^d)$, we have 
\begin{equation}
    \int g(X/\gamma) f(X)  dX 
    = \bar{g} \int f(X)  dX +
    O(\gamma^N) ~,
\label{scale-sepa-lemma1}
\end{equation}
for any $N$.
\end{prop}

\begin{prop}\label{sc-sep2} Let $g(X,Y)$ be a continuous
 function on $\R^d \times 
\R^d$ which is periodic in $Y\in \T^d$,  and denote 
$\bar g(X ) =\frac{1}{(2\pi)^d} \int_{\T^d} g(X,Y) dY $ its average value over 
$\T^d$. For any function $f(X)$  
in the Schwarz space ${\mathcal S}(\R^d)$, we have 
\begin{equation}
    \int g(X,X/\gamma) f(X)  dX 
    = \int \bar{g}(X)  f(X)  dX  
      + O(\gamma^N) ~,
\label{scale-sepa-lemma}
\end{equation}
for any $N$.
\end{prop}

The formal derivation of  the  effective equations satisfied by 
$(\zeta_0,\psi_0)$ is to write \eqref{eq4.1}\eqref{eq4.2} in the sense
of distributions, using test functions $f$ that depend only on the 
large scale variable $X$, averaging over the variables $Y$, and
neglecting all the terms that are understood 
to be of lower order (a rigorous justification of this process will be
the object of the next section). Denote the mean value over the $Y$
variables by $(\bar \zeta_a,\bar \psi_a)  = ( \frac{1}{|\Gamma|}
\int_\Gamma \zeta_a(X, Y) dY, \frac{1}{|\Gamma|}\int_\Gamma \psi_a(X,
Y) dY)$. Then 
\begin{equation} 
   \partial_t \overline{\zeta}_a 
   = \partial_t \zeta_0 + \partial_\tau \overline{\zeta_1} 
     + \sqrt{\mu} \partial_t \overline{\zeta_1} ~, 
\end{equation}
and similarly
\begin{equation} 
   \partial_t \overline{\psi}_a 
   = \partial_t \psi_0 + \sqrt{\mu}\partial_\tau \overline{\psi_1} 
     + \mu \partial_t \overline{\psi_1} ~. 
\end{equation}
We assume that $\bar \zeta_1 = 0$ and $\bar \psi_1 =0$, an assumption
that will be shown to be consistent with equations \eqref{eq4.10} and
\eqref{eq4.11} derived  below. We obtain therefore
$\overline{\zeta}_a\simeq\zeta_0$ and  $\overline{\psi}_a\simeq\psi_0$
at lowest order. For equation \eqref{eq4.1}, using that the
$Y$-derivative of a periodic function of $Y$ has mean value zero, we
find that    
\begin{equation}\label{eq4.6}
   \partial_t \zeta_0 
   = - h_0 \Delta \psi_0 - \nabla\zeta_0\cdot \nabla \psi_0 ~. 
\end{equation}
>From equation \eqref{eq4.2}, we find
\begin{equation}\label{eq4.7}
   \partial_t\psi_0 + (\zeta_0 +\sqrt{\mu}\bar  \zeta_1)
  + \frac{1}{2} |\nabla \psi_0|^2 = 0 ~.
\end{equation}
The lowest order approximation to the system
\eqref{eq4.6}\eqref{eq4.7} takes the form of the classical shallow
water system (with $V_0 = \nabla\psi_0$ and $h_0=1+\zeta_0$), namely 
\begin{equation} \label{eq4.8}
   \left\lbrace	\quad \begin{array}{l}
   \partial_t \zeta_0 = - h_0 \nabla\cdot V_0 -\nabla\zeta_0\cdot V_0\\
   \partial_t V_0 + \nabla \zeta_0 + V_0\cdot\nabla V_0 = 0 
	\end{array}\right. ~.
\end{equation}
Returning to \eqref{eq4.1}\eqref{eq4.2} and using the equations 
\eqref{eq4.8} satisfied by $(\zeta_0,\psi_0)$, we obtain at the next
order of approximation the equation 
\begin{equation}\label{eq4.10}
   \dsp \partial_\tau \zeta_1 + 
   V_0 \cdot \nabla_Y \zeta_1 = 
      |D_Y| \tanh(h_0 |D_Y|) \psi_1 
     + V_0 \cdot \nabla_Y \sech(h_0 |D_Y|)b ~.
\end{equation}
If $\bar \zeta_1=0$ at time $\tau=0$, it remains so for all times.
For the evolution equation for $\psi_1$, we find
\begin{equation}\label{eq4.11}
   \dsp \partial_\tau \psi_1 + V_0 \cdot \nabla_Y\psi_1 + \zeta_1 = 0 ~. 
\end{equation}
Again, if  $\bar \psi_1=0$ at time $\tau=0$, it remains so for all times.

The result is that \eqref{eq4.8} is the shallow water system 
\eqref{shallow-water-eqs} for $(\zeta_0, V_0)$, with $V_0=\nabla
\psi_0$, and the dispersive corrections are given by 
\eqref{eq4.10}\eqref{eq4.11}. This derivation, and a rigorous
justification of it, are the principal subject of this paper.


\subsection{Regularity of the approximate solution}

The approximate solution is constructed from the solution of a system of
simpler model equations \eqref{eq4.8} and \eqref{eq4.10}\eqref{eq4.11} 
where the first is a version of the classical shallow water equations. 
The second is the system for linear water waves in the rapid variables 
$(Y,\tau)$, with a forcing term  due to the presence of bottom
variations, whose coefficients depend upon $(\zeta_0(t,X), \psi_0(t,X))$, 
for which the slow variables are considered as being fixed. 

\begin{thm} \label{prop4.3}
For $r>d/2 +1$, given initial data $(\zeta_0(\cdot,0), V_0(\cdot,0))$
in $H^r(\R^d)\times H^{r}(\R^d)^d$ such that \eqref{eq12bis} is satisfied.  
Then there exists $T>0$ and a smooth solution  
$(\zeta_0, V_0)$ in $C([-T,T];H^r(\R^d)\times H^{r}(\R^d)^d)$ to
\eqref{eq4.8} with this initial data. 
\end{thm}

\begin{proof}
The shallow water system can be written as a symmetric hyperbolic
system for the vector function $(\zeta_0,V_0 = \nabla\psi_0)$. For Sobolev 
index $r>d/2 +1$, these equations are locally well posed in time 
for $(\zeta_0(\cdot,0), V_0(\cdot, 0)) \in H^r\times H^{r}$ satisfying 
the condition \eqref{eq12bis} (see for example \cite{John76}).
\end{proof}

The components of the corrector $(\zeta_1,\psi_1)= (\zeta_1(\tau,Y;
t,X), \psi_1(\tau,Y;t,X))$ are multiple scale functions, satisfying a
system of the form 
\begin{equation}\label{eq4.12}
	\partial_\tau\left(\begin{array}{c}
	 \zeta_1  \\
	 \psi_1
	\end{array}\right) 
        + V_0(t,X) \cdot \nabla_Y
	\left(\begin{array}{c}
	 \zeta_1  \\
	 \psi_1
	\end{array}\right) +
	\left(\begin{array}{cc} 
	0 & -\abs{D_Y} \tanh(h_0(t,X) |D_Y| ) \\ 
	I & 0\end{array}\right)
      \left(\begin{array}{c}
	 \zeta_1  \\
	 \psi_1
	\end{array}\right) =
      \left(\begin{array}{c}
	f \\
	g
	\end{array}\right) ~,
\end{equation} 
for which the large scale variables $(t,X)$ enter as parameters.
In the case of equations \eqref{eq4.10}\eqref{eq4.11} the
inhomogeneous forcing functions are given by
\begin{equation}\label{eq4.12bis}
  f(Y;t,X) = V_0(t,X) \cdot
  \nabla_Y\sech(h_0(t,X)|D_Y|)b(Y),\qquad g=0,
\end{equation}
(in particular, it is autonomous, namely independent of $\tau$, and
its zero Fourier mode $\overline{f}$ vanishes).\\
Initial data for this system is given in the form 
$(\zeta_1(0,\cdot \, ; t,X), \psi_1(0,\cdot \, ;t,X))$ in
$H^r_Y(\T^d)\times H^{r+1/2}_Y(\T^d)$ with zero mean on $\T^d$. 
The dependence of these solutions on the variables $(t,X)$ will be
quantified in a paragraph below.

For the sake of clarity, we omit the dependence on the variables $(t,X)$ in the statement below, 
since they only act as parameters in \eqref{eq4.12}. It is also convenient to introduce the energy or order $r$ ($r\in \R$) defined for 
all couple of function $(u,v)$ on $\T^d$ with zero mean by
\begin{equation} \label{energy}
  \|(u,v) \|_{E^r} ^2 = \sum_{k\ne0} (1+k^2)^r( |\hat u_{k}|^2 +\abs{k}\tanh(h_0(t,X)\abs{k})  |\hat v_{k}|^2) ~,
 \end{equation}
where $\widehat{u}_k$ and $\widehat{v}_k$ stand for the Fourier components of $u$ and $v$.
\begin{thm} \label{prop4.4}
Let $r\in \R$. For all $(f,g) \in C({\mathbb R} : H^r_Y\times H^{r+1/2}_Y(\T^d))$ with zero mean 
 and all 
$(\zeta_1(0,\cdot), \psi_1(0,\cdot)) 
   \in H^r_Y(\T^d)\times H^{r+1/2}_Y(\T^d)$ with zero mean,
there exists a unique solution   
$(\zeta_1, \psi_1)$ of \eqref{eq4.12} in
$ C({\mathbb R} : H^r_Y(\T^d)\times
H^{r+1/2}_Y(\T^d))$ with initial values given by 
$(\zeta_1(0,\cdot), \psi_1(0,\cdot))$. Moreover, this solution has zero mean and one has
$$
\forall \tau\in\R,\qquad  \|(\zeta_1(\tau),\psi_1(\tau)) \|_{E^r} ^2
\leq  \|(\zeta_1(0),\psi_1(0)) \|_{E^r} ^2 +\tau  \sup_{0\leq \tau'\leq \tau}\|(f(\tau'),g(\tau') \|_{E^r} ^2.
$$
\end{thm}

This theorem implies that $(\zeta_1,\psi_1)$ is bounded in
$H^r\times H^{r+1/2}$ over any time interval
$\tau \in [-T_1,T_1]$. However, this bound may grow
as $T_1 \to \infty$ due to the possible presence of secular
terms. Furthermore, when considering the dependence of this solution
on the parameters $(t,X)$, secular growth of the quantities
$\partial_X(\zeta_1,\psi_1), \partial_t(\zeta_1,\psi_1)$ is quite
possible, and would affect the validity of the solution decomposition  
\eqref{Eqn:BasicAnsatz1}\eqref{Eqn:BasicAnsatz2} over long time
intervals $\tau \in [-T/\gamma, T/\gamma]$. In Theorem \ref{thmbragg},
we show that such effects do not occur, at least in the absence of
Bragg resonances, and for initial data
$(\zeta_1(0,\cdot),\psi_1(0,\cdot))$ chosen to be stationary in the
local environment defined by $(\zeta_0,\psi_0)$.  

\begin{proof}
For the sake of simplicity, we take $g=0$ in the proof below; the adaptation to the general case is straightforward.
Considered independently of the large scale variables $(t,X)$, the
system of equations \eqref{eq4.12} over $Y \in \T^d$ has constant
coefficients, and the solution operator can be conveniently expressed
in Fourier transform. The evolution of the individual Fourier modes is
described by the system 
\begin{equation}\label{eq4.12a}
	\partial_\tau\left(\begin{array}{c}
	 \hat \zeta_{1k}\\
	 \hat\psi_{1k}
	\end{array}\right) 
     +  \left(\begin{array}{cc} 
	ik\cdot V_0 & 0 \\ 
	0 & ik\cdot V_0 \end{array}\right)
       \left(\begin{array}{c}
       \hat \zeta_{1k}\\
	 \hat\psi_{1k}
	\end{array}\right)
     +  \left(\begin{array}{cc} 
	0 & - \omega_k^2 \\ 
	I & 0\end{array}\right)
     \left(\begin{array}{c}
     \hat \zeta_{1k}\\
	 \hat\psi_{1k}
	\end{array}\right) 
   =  \left(\begin{array}{c}
	\hat f_{k} \\
	0
	\end{array}\right) ~,
\end{equation}
where the local velocity is $V_0(t,X)$, and the frequency is given by 
$\omega_k=\big(|k| \tanh(h_0(t,X) |k|)\big)^{1/2}$, both of which depend
parametrically upon the long scale spatial variable $(t,X)$. \\
Defining new coordinates $u_k := \omega_k^{-1/2}\hat\zeta_{1,k}$ and  
$v_k := \omega_k^{+1/2}\hat\psi_{1,k}$, the propagator for
\eqref{eq4.12}\eqref{eq4.12a} is given by 
\[
     \exp\left[ \tau \left(\begin{array}{cc} 
	-ik\cdot V_0 & 0 \\ 
	0 & -ik\cdot V_0 \end{array}\right)
        + \tau \left(\begin{array}{cc} 
	0 &  \omega_k \\ 
	- \omega_k & 0 \end{array}\right) 
        \right] 
     = e^{-ik\cdot V_0\tau}
         \left(\begin{array}{cc} 
	 \cos(\omega_k\tau) & \sin(\omega_k\tau) \\ 
	-\sin(\omega_k\tau) & \cos(\omega_k\tau) \end{array}\right) ~.
\]
Using complex notation for this system, define
\[
    Z_k := u_k + iv_k ~, \qquad W_k := u_k - iv_k ~,
\]
with which we express the general solution to \eqref{eq4.12a};
\begin{eqnarray}\label{Eqn:Quadrature1}
   && Z_k(\tau) = e^{-i\tau[\omega_k + k\cdot V_0]} Z_k(0) 
     + \int_0^\tau e^{-i(\tau-s)[\omega_k + k\cdot V_0]} \omega_k^{-1/2}
     \hat{f}_k(s) \, ds \\
   && W_k(\tau) = e^{+i\tau[\omega_k - k\cdot V_0]} W_k(0) 
     + \int_0^\tau e^{+i(\tau-s)[\omega_k - k\cdot V_0]} \omega_k^{-1/2}
     \hat{f}_k(s) \, ds ~.   \nonumber 
\end{eqnarray}
Standard use of the Plancherel identity implies that
\begin{eqnarray*}
  &&  \|Z(\tau,\cdot)\|_{H_y^{r+1/4}}^2 \leq \|Z(0,\cdot)\|_{H_y^{r+1/4}}^2 
   + C_0 |\tau| \|f(s,\cdot)\|_{L^\infty_s([-\tau,\tau] :
     H_y^{r})}^2 \\
  &&  \|W(\tau,\cdot)\|_{H_y^{r+1/4}}^2 \leq \|W(0,\cdot)\|_{H_y^{r+1/4}}^2 
   + C_0 |\tau| \|f(s,\cdot)\|_{L^\infty_s([-\tau,\tau] :
     H_y^{r})}^2 ~,
\end{eqnarray*}
where we have used that $\omega_k \simeq \langle k \rangle^{1/2}$.
Recovering our original variables
\begin{eqnarray*}
  && \frac{1}{2}(Z_k + W_k) = u_k = \omega_k^{-1/2} \hat\zeta_{1k} ~, \quad
     \frac{1}{2i}(Z_k - W_k) = v_k = \omega_k^{+1/2} \hat\psi_{1k} ~,
\end{eqnarray*}
the result is as stated in Theorem~\ref{prop4.4}, with in addition a
quantitative estimate on the growth in the fast time variable $\tau$.
\end{proof}
 
\subsection {The Bragg resonance condition.}

Solutions to the linear equation \eqref{eq4.12} exist for all $\tau
\in \R$, however $(\zeta_1(\cdot,\tau), \psi_1(\cdot,\tau))$ may exhibit 
secular growth in time; more precisely, it may grow linearly with respect to $\tau$.
This is a concern for our model system because 
$\tau = t/\gamma$ is the rapid timescale, therefore over physically relevant time
intervals of $\BigOh{1}$ in the slow time variable $t$, solutions 
of \eqref{eq4.12} may grow from $\BigOh{1}$ quantities to
$\BigOh{1/\gamma}$  quantities, thus leaving the range of validity of
our assumption regarding the asymptotic scaling regime. This secular
growth for Fourier modes $(\hat\zeta_{1,k},\hat \psi_{1,k})(\tau)$ is due to the presence of 
Bragg resonances of the $k^{th}$ Fourier mode of the corrector
solution with the periodic variations of the bottom topography defined
by $b(Y)$ for which $\hat b_k \ne 0$. Note that these resonances differ from the classical Bragg resonances which are obtained with surface waves and bottoms of comparable wavelength \cite{LiuYue}; to our knowledge they had not been exhibited before. Such a resonance occurs at time $t$ and in $X$ if for some $k\neq 0$ such that $\hat b_k\neq 0$, 
\begin{equation}\label{BraggResonance}
   \omega_k(X,t)^2 = \big(k\cdot V_0(X,t)\big)^2,
\end{equation} 
where $V_0 = \nabla \psi_0$, $\omega_k=\big(\abs{k}\tanh(h_0\abs{k})\big)^2$.\\
In absence of such resonances, it is quite easy to check that there is no secular growth of the first corrector: $(\zeta_1(\tau,\cdot;t,X),\psi_1(\tau,\cdot;t,X))$ remains bounded with respect to $\tau$ in the energy norm (\ref{energy}). This easily follows from the fact that $(\zeta_1,\psi_1)$ solves (\ref{eq4.12}) with a forcing function $f$ given by (\ref{eq4.12bis}) which is independent of $\tau$. The time integral in (\ref{Eqn:Quadrature1}), which is at the origin of the secular growth, can then be explicitly computed, and it is obviously bounded in absence of Bragg resonance.\\
Controlling the error corresponding to our approximation requires however some bound on 
the parametric derivatives  of $\zeta_1$ and $\psi_1$ (i.e. their derivatives with respect to
$t$ and $X$). These parametric derivatives also solve a problem of the form (\ref{eq4.12}), albeit with a different forcing term $(f,g)$. Theorem \ref{prop4.4} can therefore be used to give some control on the energy norm of these
parametric derivatives; however the forcing term now depends on $\tau$ and the linear secular growth in $\tau$ that appears in the estimate of Theorem \ref{prop4.4} cannot be removed as above.\\
As previously explained, this secular growth is destructive for our approximation. While it cannot be avoided in general for the parametric derivatives of $(\zeta_1,\psi_1)$, we still have some freedom to eliminate  it. Indeed, the choice for the initial condition associated to (\ref{eq4.12}) is so far completely arbitrary. It turns out that, in absence of Bragg resonance,  there is \emph{one} choice of initial data that removes the secular growth.
This removal is quite spectacular since the corresponding solutions are \emph{independent} of $\tau$ (and therefore bounded together with all their parametric derivatives). These constant solutions are found by removing the
$\tau$-derivative in (\ref{eq4.12})-(\ref{eq4.12bis}) which leads to solving the following problem,
$$
	\left(\begin{array}{cc} 
	V_0\cdot \nabla_Y & -\abs{D_Y} \tanh(h_0 |D_Y| ) \\ 
	I &  V_0\cdot \nabla_Y
 \end{array}\right)
      \left(\begin{array}{c}
	 \zeta_1  \\
	 \psi_1
	\end{array}\right) =
      \left(\begin{array}{c}
	V_0 \cdot
  \nabla_Y\sech(h_0|D_Y|)b \\
0
	\end{array}\right) ~.
$$
In absence of Bragg resonance, this yields, for all $k\in \Z$,
\begin{equation}\label{sollocallycst}
\hat \zeta_{1,k}= -\frac{(V_0\cdot k)^2\sech(h_0\abs{k})}{-(V_0\cdot k)^2+\abs{k}\tanh(h_0\abs{k})}\hat b_k,
\qquad
\hat \psi_{1,k}=-i\frac{(V_0\cdot k)\sech(h_0\abs{k})}{-(V_0\cdot k)^2+\abs{k}\tanh(h_0\abs{k})}\hat b_k.
\end{equation}

A quantitative measure of nonresonance 
with respect to a sequence $\{ 0 < B_k < +\infty \, : \, k \in \Z^d \}$ 
is necessary for the analysis that follows; the $k^{th}$ Fourier modes
$(\hat \zeta_{1,k},\hat \psi_{1,k})$  are {\it nonresonant} at $(X,t)$ with respect to the
homogenized solution $(\zeta_0(\cdot,t), \psi_0(\cdot,t))$ and the
bottom topography $b(Y)$ if $\hat{b}_k \not= 0$ and one has  
\begin{equation}\label{BraggNonresonance}
   | \omega_k(X,t)^2 - \big(k\cdot V_0(X,t)\big)^2| > \frac{1}{B_k} ~.
\end{equation} 
The sequence $\{B_k \}$ is effectively a bound
on the small divisor condition governing Bragg resonances,
locally in $(X,t)$. Given a sequence $\{ B_k \, : \, k \in \Z^d \}$
such that $B_k < e^{{\bar h}|k|/2}/\delta$, if \eqref{BraggNonresonance} 
holds for all $k \not= 0$, a local stationary solution exists, and
furthermore the secular growth of local solutions can be controlled
locally in $(X,t)$. When \eqref{BraggNonresonance} 
holds uniformly in $(X,t)$ for all $\hat b_k \ne 0$, it is a
nonresonant situation (relative to the small divisor conditions 
$\{B_k \}$) and solutions can be controlled globally. 
\begin{thm}\label{thmbragg}
Let $r\in {\mathbb N}$, $r'>d/2+r+1$, $T>0$ and $(\zeta_0, V_0)\in C^r([-T,T];H^{r'-r}(\R^d)^{1+d})$ 
be a solution of the shallow water 
equations \eqref{eq4.8}, such that
$$
\exists \alpha_0>0,\qquad \forall (X,t)\in \R^d\times [-T,T],\quad 1+\zeta_0(X,t)\geq \alpha_0.
$$
Assume also that the nonresonance condition (\ref{BraggNonresonance}) holds with  $B_k <
e^{{\bar h}|k|}/\delta$ (for some $\delta>0$ and $0<\bar h<\alpha_0$). Then  there exists a
unique locally stationary solution
of \eqref{eq4.12}, which is given by (\ref{sollocallycst}). In particular, one has, for all $0\leq s \leq r'-r$ and all 
$s'>0$,  
\begin{eqnarray}
  && |\zeta_1|_{C^r([-T,T];H_X^{s}\times H_Y^{s'})} 
     + |\psi_1(\cdot)|_{C^r([-T,T];H_X^{s}\times H_Y^{s'})} \\ 
  &&\qquad \leq C_{rss'}(|\zeta_0|_{C([-T,T]; C^{r}_X)},|\psi_0|_{C([-T,T];C^{r}_X)})
       (|\zeta_0|_{C^r([-T,T]; H_X^s)} +|\psi_0|_{C^r([-T,T]; H_X^s)}) 
       |b|_{L^2_Y}~.  \nonumber
\end{eqnarray}
\end{thm}

\begin{proof}
Let $F_k:= \R\times\R^d\mapsto \R$ defined as
$$
F_k(V,\zeta)=-\frac{(V\cdot k)^2\sech((1+\zeta)\abs{k})}{-(V\cdot k)^2+\abs{k}\tanh((1+\zeta)\abs{k})},
$$
so that $\hat\zeta_{a,k}=F(V_0,\zeta_0)\hat b_k$. The mapping $F_k$ is smooth, vanishes at the origin, and is rapidly 
decaying together with
all its derivatives as a consequence of the nonresonance assumption (\ref{BraggNonresonance}). It follows therefore from Moser's estimate that
$$
\abs{k}^{s'}\abs{\hat\zeta_{1,k}}_{C([-T,T];H^{s}_X)}\leq C(s,s',\abs{\zeta_0}_{C([-T,T];L^\infty_X)},\abs{V_0}_{C([-T,T];L^\infty_X)})
\big(\abs{\zeta_0}_{C([-T,T];H^s_X)}+\abs{V_0}_{C([-T,T];H^s_X)}\big)\abs{\hat b_k},
$$
so that the bound given in the lemma stems from Plancherel's inequality in the case $r=0$. Bounds for $r>0$ and 
on $\psi_1$ are obtained in the same way.
\end{proof}

The question as to how often Bragg resonances occur merits a
discussion. For $k\not= 0$ fixed, \eqref{BraggNonresonance} is an 
open condition on the state parameters $(\zeta_0,V_0)\in C^0$. In two 
dimensions (that is, when $d=1$), it is related to the local Froude 
number of the flow, defined by 
\[
     F\! r^2(X,t) := \frac{V_0^2(X,t)}{h_0(X,t)} ~;
\]
indeed, the nonresonance condition \eqref{BraggNonresonance} can be stated equivalently as
$$
Fr^2(X,t)=\frac{\tanh(h_0(X,t)|k|)}{h_0(X,t)|k|}
$$
and in particular for supercritical flows $F\! r^2(X,t) \geq 1$, Bragg
resonances are absent. However, the Froude number is an indication of 
criticality which is local in $X$, and because $V_0 \in H^r$,
solutions can be supercritical  only on compact sets. 
For subcritical flows, and for $(X,t)$ fixed, at most one 
Fourier mode can be in resonance\footnote{This follows
from the fact that $\tanh(x)/x$ is strictly decaying on $\R^+$.}.
If $b(Y)$ is a trigonometric polynomial, then there are a finite number of 
resonances. Any  2-resonances are separated by a region of non-resonance,
and for $V_0 \to 0$ as $X\to \pm \infty$ further resonances are
avoided. In particular, if $k_{max}$ denotes the highest nonzero Fourier
mode of $b$ and $k_{min}$ the lowest one, then  Bragg resonances are possible only if
\[
   Fr^2_{min}\leq Fr^2\leq Fr^2_{max}, \quad 
   \mbox{ with }\quad  Fr^2_{min}=\frac{\tanh(h_0(X,t)|k_{min}|)}{h_0(X,t)|k_{min}|}, \quad
Fr^2_{max}=\frac{\tanh(h_0(X,t)|k_{max}|)}{h_0(X,t)|k_{max}|}.
\]

For general $b(Y)$ with infinitely many nonzero  Fourier coefficients,
any zero of velocity $V_0$ is a point of accumulation of resonances 
and in particular, small resonant patches will appear because 
$V_0 \to 0$ as $X\to \pm \infty$. Their asymptotic strength is related
to the large $k$- behavior of $|\hat b_k|$.

The character of resonance for $d\geq 2$ is different.
For $b(Y)$ given by a trigonometric polynomial, resonances are isolated 
for the same reason as for $d=1$. In the case of  a general $b(Y)$,
 there is the potential for a dense set of resonances
in the state space $(\zeta_0,V_0)$ and not just at $V_0=0$. This
can be seen through the parametric dependence on $(\zeta_0,V_0)$ of the
resonant condition \eqref{BraggResonance} in wavenumber space. Given
 $(\zeta_0,V_0)$, this condition defines a hypersurface $E_k$ in
 $ k \in  \Z^d$, which, if it passes through a lattice point $ k \in
\Z^d$ with $\hat b_k \ne 0$,  gives rises to a resonance. Even if it
does not intersect a lattice point, under arbitrarily small
perturbations at $(\zeta_0,V_0)$, it will. Hence the set of resonant
states $(\zeta_0,V_0)$ is dense. 

Nonetheless, in the measure 
theoretic sense, Bragg resonances are relatively rare. That is, 
there is a set of states  $(\zeta_0,V_0)$ for which 
\eqref{BraggNonresonance} is satisfied for all $k \not= 0$, 
such that its complement has measure less than $C\delta$. 
Indeed, fix $k \not= 0$. The gradient of the resonance condition with 
respect to $(\zeta_0,V_0)$ on the curve $E_k$  is non-vanishing, and is
of amplitude of order $\BigOh{|k|}$. Thus state variables
$(\zeta_0,V_0)$ of distance $(B_k|k|)^{-1}$ from $E_k$ will satisfy
the nonresonance condition \eqref{BraggNonresonance}. The union over
$k\not= 0$ of $(B_k|k|)^{-1}$-tubular neighborhoods of the sets $E_k$
consists of the `bad' states, for which there  exists at least
one near resonance as in \eqref{BraggNonresonance}. This union 
has relative measure bounded above by
\[
    \sum_{k\not= 0} \frac{1}{|k|B_k} < C\delta ~.
\]
Therefore the resonant set is dense, but has relatively small 
measure in the space of states $(\zeta_0,V_0)$. Moreover, we see that 
the set for which \eqref{BraggNonresonance} is satisfied 
has the character of a Cantor set. 


\subsection{Consistency analysis}
\label{consistency-sect}
The purpose of this section is to evaluate the error that is made 
when approximating the solution ($\zeta,\psi$)  of the full water 
wave problem by the functions \eqref{ansatz}, where the components 
($\zeta_0,\psi_0$) and ($\zeta_1,\psi_1$) satisfy the effective system 
of equations \eqref{eq4.7} and \eqref{eq4.10} respectively.

Write the full water wave problem \eqref{eq7} as 
\begin{equation}\label{Eqn:ConsistencyError}
\begin{array}{l} 
     E_1(\zeta,\psi)=0 ~,   \\
     E_2(\zeta,\psi) = 0 ~.
\end{array}
\end{equation}
where $E_1$ and $E_2$ identify with the LHS of equations \eqref{eq7}.
We  denote our construction of an approximate solution by 
$E_a := (E_1(\zeta_a,\psi_a),E_2(\zeta_a,\psi_a))$, where 
$(\zeta_a,\psi_a)$ is defined in 
\eqref{Eqn:BasicAnsatz1}\eqref{Eqn:BasicAnsatz2}. For this approximate
solution, the error is given by the expression

\begin{eqnarray} 
  &&  E_1(\zeta_a,\psi_a)=- \sqrt{\mu} (G \psi_a)_{res} -\sqrt{\mu} \partial_t \zeta_1 
     \label{Eqn:FirstComponentOfError} \\
  &&  E_2 (\zeta_a,\psi_a)= 
     \mu \nabla \psi_0\cdot \nabla_X\psi_1 +
     \frac{\mu}{2} 
     | \nabla_Y\psi_1 +\sqrt{\mu} ~\nabla_X \psi_1|^2
-\mu \partial_t \psi_1 \label{Eqn:SecondComponentOfError} \\
   && \qquad -\mu\frac{ \Big(  \frac{1}{\mu} 
   \big( (G \psi_a)_{\hbox{\small \itshape eff}} +  \mu^{3/2} (G \psi_a)_{res} \big)
      + (\nabla \zeta_0 + \nabla_Y\zeta_1 + \sqrt{\mu}\nabla_X\zeta_1)\cdot
      (\nabla \psi_0 +\sqrt{\mu} \nabla_Y\psi_1 + \mu \nabla_X\psi_1)
      \Big)^2}{2 \big(1 + 
      \mu|\nabla\zeta_0 + (\nabla_Y + \sqrt{\mu}
      ~\nabla_X)\zeta_1|^2\big) } ~. \nonumber  
\end{eqnarray}
The statement that the expression $(\zeta_a, \psi_a)$ is a good 
approximation for the equations \eqref{eq7} is that the error is 
small, in an appropriate norm, for small $\mu$. The 
Theorem~\ref{Thm:OfConsistency} is a result of this form, implying 
the consistency of the approximate solution.  
We recall that the leading term $(\zeta_0,\psi_0)$ of the approximation solves the nonlinear shallow water equations (\ref{shallow-water-eqs}) while, in absence of Bragg resonances, the correctors $(\zeta_1,\psi_1)$ are explicitly given by (\ref{sollocallycst}).
\begin{thm}\label{Thm:OfConsistency} Under the assumptions of Theorem \ref{thmbragg}, the approximate solution 
$(\zeta_a, \psi_a)$ given by expression 
\eqref{Eqn:BasicAnsatz1}\eqref{Eqn:BasicAnsatz2} satisfies
the following consistency estimates
\begin{equation}
\begin{array}{l}
     |E_1(\zeta_a,\psi_a)|_{L^2} \leq  C_a \mu^{3/8} ~, \\
     |E_2(\zeta_a,\psi_a)|_{H^{1/2}} \leq C_a \mu^{3/4}  ~,
\end{array}
\end{equation}
for the error term for the water wave system \eqref{eq7}. The
constant $C_a$ is of the form
$$
C_a=C(\frac{1}{\alpha_0}, \abs{\zeta_0}_{C^4} ~, \, \abs{V_0}_{H^4} ~,|b|_{L^2}).
$$
\end{thm}

\begin{rema}
The quantities $(\zeta_0,\psi_0,\zeta_1,\psi_1)$ are solutions of the model equations 
and  following  Theorems \ref{prop4.3}, \ref{prop4.4} and \ref{thmbragg}, are bounded 
along with their derivatives in terms of the initial data.
The norm in which the error is measured 
is relatively weak,  the reason being that we are dealing with a problem 
with fast oscillating functions. It is however a natural norm for this problem since it coincides with the norm of the energy functional associated to the water waves equations. 
\end{rema}

\begin{proof}
The first component $E_1$ satisfies
\[
   |E_1(\zeta_a,\psi_a)|_{L^2} \leq \sqrt{\mu}
    |(G \psi)_{res}|_{L^2}\sqrt{\mu} |\partial_t\zeta_1|_{L^2} .
\]
By Proposition~\ref{Lemma3.7}, the norm $|(G \psi)_{res}|_{L^2}$ is bounded
by  $\mu^{-1/8}$. This estimate involves norms of $\zeta_1$, $\psi_1$ and their derivatives that can becontroled using Theorem \ref{thmbragg} in terms of norms of the leading term $\zeta_0$ and $\psi_0$. 
Thus
the estimate of the first component of $E_a$ is shown to be as stated
in the theorem.

The second component of the error $E_2$ is given in 
\eqref{Eqn:SecondComponentOfError}. It is made up of
a complicated nonlinear
expression. This nonlinear quantity consists of several types of
terms, distinguished by whether they depend upon surface variables
alone, or whether they depend upon the Dirichlet -- Neumann operator
and thus on the solution of an elliptic boundary value problem with
oscillatory coefficients.  
Further terms of the RHS depend only upon surface variables,
as products of functions of 
$X$ and/or multiscale functions.  We will use the interpolation estimates in the
 form (for the sake of clarity, we omit the dependence on time here)
\begin{equation}
|f|_{H^{1/2}} \leq \mu^{-1/4} |f|_{L^2} +\mu^{1/4} |f|_{H^1}
\end{equation}
 Products such as 
$|\nabla\psi_0(X)\cdot \nabla_Y\psi_1(X,X/\sqrt{\mu})|_{H^{1/2}}$
are controlled by
\begin{equation}
   \mu |\nabla\psi_0 \cdot \nabla_Y\psi_1|_{H^{1/2}}  
   \leq   \mu^{3/4}  |\nabla\psi_0|_{C^1} |\nabla_Y\psi_1|_{H^{1,r_0+1}}
\end{equation}
Products of  multiscale functions are bounded by 
\begin{equation}
\mu \big\vert |\nabla_Y\psi_1|^2\big\vert_{H^{1/2}} 
   \leq \mu^{3/4} |\nabla_Y\psi_1|_{C^0_{XY}} |\nabla_Y\psi_1|_{H^{0,r_0+1}}
\end{equation}
Other terms in the first line of the  RHS of (\ref{Eqn:SecondComponentOfError}) 
are bounded similarly.
We now turn to the second line of the RHS of  \eqref{Eqn:SecondComponentOfError}.
It has the form $\mu \frac{A}{B}$ that we need to bound in $H^{1/2}$ norm.
The
denominator $B$ satisfies $ B=2(1 + |\nabla\zeta_0 + \mu(\nabla_Y 
   + \sqrt{\mu}\nabla_X)\zeta_1|^2)\geq 2$ and we can therefore write
\begin{equation}
     \mu|\frac{A}{B}|_{H^{1/2}} 
  \leq \mu ( \mu^{-1/4} |\frac{A}{B}|_{L^2}+ \mu^{1/4} |\frac{A}{B}|_{H^1})
  \leq \mu^{3/4}  |A|_{L^2} + \mu^{5/4}   |\nabla A|_{L^2} + \mu^{5/4} |A \nabla B|_{L^2}.  
\end{equation}

The numerator $A$ contains many terms. To bound its $L^2$ norm, we have 
for example terms of the form
\begin{equation}
 |(\frac{1}{\mu} G\psi)_{eff})^2 |_{L^2} \leq  C
 \end{equation}
where $C$ depends on  $|\nabla\zeta_0|_{C^2}$, $|\nabla\psi_0|_{C^2}$,
$|\nabla_{X,Y} \zeta_1|_{H^{3,r_0+1}}$, $|\nabla_{X,Y} \psi_1|_{H^{3,r_0+1}}$. 
Here again, Theorem \ref{thmbragg} is
used to control the last two quantities in terms of norms of $\zeta_0$ 
and $\psi_0$.
 Estimates of terms of the numerator $A$  which depend upon the quantity
$(G\psi)_{res}$ from the Dirichlet -- Neumann operator use
the results of Section~3 on the boundary value problem with periodic
oscillatory coefficients. For example, 
\begin{equation}
 |(\mu^{1/2}( G\psi)_{res})^2 |_{L^2} \leq  C\mu^{3/4}
 \end{equation}

Examination of all terms leads to 
$|A|_{L^2} \leq C$.
Noting that the computation of $\nabla B$ gives one factor $\mu$ and that each derivation costs
a factor $\mu^{1/2}$, we find 
\begin{equation}
|\nabla B|_{L^4} \leq C \mu^{1/4}
\end{equation}
Considering all terms similarly, we arrive to  the conclusion of Theorem \ref{Thm:OfConsistency}.
\end{proof}

\noindent
{\bf Acknowledgements.} 
W. C. is partially supported by the Canada 
Research Chairs Program
and NSERC through grant number 238452--06, 
D.L.  by the ANR project MathOcean  
ANR-08-BLAN-0301-01 and  
C. S. by NSERC through grant number 46179--05.
D. L. would like to thank  McMaster
University, and W. C and C.S. the Ecole Normale Sup\'erieure 
for their  hospitality during the course of this work.


\end{document}